\documentclass[10pt,reqno]{amsart}
\usepackage{bbm}
\usepackage{amsmath,amsthm,amsfonts,amssymb,amscd,mathrsfs}
\usepackage{psfrag,graphicx}

\usepackage{epic,eepic}
\usepackage{color}
\usepackage{ebezier}

\theoremstyle{plain}

\newtheorem{Thm}{Theorem}[section]
\newtheorem{Lem}{Lemma}[section]
\newtheorem{Prop}{Proposition}[section]

\newtheorem{Rmk}{Remark}[section]

\newtheorem{Def}{Definition}[section]

\newcommand{\N}{\ensuremath{\mathbb{N}}}
\newcommand{\R}{\ensuremath{\mathbb{R}}}

\long\def\begcom#1\endcom{}

\newcommand{\Leb}{\operatorname{vol}}
\newcommand{\diam}{\operatorname{diam}}

\newcommand{\length}{\operatorname{\length}}

\newcommand{\vol}{\operatorname{vol}}

\newcommand{\loc}{\operatorname{loc}}

    \def\ed{\end{description}}

\def\supp{\operatorname{supp}}

\def\length{\operatorname{length}}
\def\loc{\operatorname{loc}}

\def\dim{\operatorname{dim}}
\def\vep{\operatorname{\varepsilon}}

\begin{document}

\title[Entropy formula of folding type for $C^{1+\alpha}$ maps]
{Entropy formula of folding type for $C^{1+\alpha}$ maps}
\author[G. Liao]{Gang Liao$^\dag$} 

\author[S. Wang]{Shirou Wang$^{\ddag}$}

\keywords{Entropy formula,  folding entropy, absolute continuity, degeneracy}

{
\thanks{2020 {\it
 Mathematics Subject Classification}.  37C40, 37A60, 37D25}

\thanks{$\dag$
	School of Mathematical Sciences,  Center for Dynamical Systems and Differential Equations, Soochow University, Suzhou 215006, China. {\it Email: lg@suda.edu.cn}}

\thanks{$\ddag$ School of Mathematics, Jilin University, Changchun 130012, China, and Department of Mathematical $\&$ Statistical Sciences, University of Alberta, Edmonton, AB T6G 2G1, Canada. {\it  Email: shirou@jlu.edu.cn}. Corresponding author}

\begin{flushleft}
	{\it Comm. Math. Phys.},  to appear
	
	\vspace{5mm}
\end{flushleft}

\thanks{G. Liao was partially supported by the National Key R\&D Program of China (2022YFA1005802), NSFC (11790274, 12071328, 12122113), BK20211595, and Tang Scholar. S. Wang was partially supported by NSFC 12201244, NSFC 12271204, and a faculty development grant from Jilin University.}

\thanks{{\bf Data Availability}: Data sharing not applicable to this article as no datasets were generated or analyzed during the current study.}

\thanks{{\bf Conflicts of Interest}: All authors declare  no conflicts of interest. }

\begin{abstract}
In the study of non-equilibrium statistical mechanics, 
Ruelle derived explicit formulae for entropy production of smooth dynamical systems. The vanishing or strict positivity of entropy production is determined by the {\it entropy formula of folding type} 
\[h_{\mu}(f)= F_{\mu}(f)-\displaystyle\int\sum\nolimits_{\lambda_i(x)<0} \lambda_i(x)d\mu(x),
\]
which relates the metric entropy, folding entropy and negative Lyapunov exponents. 
This paper establishes the formula for all inverse SRB  measures of 
$C^{1+\alpha}$ maps, including those with degeneracy (i.e., zero Jacobian). More specifically, we establish the equivalence that $\mu$ is an inverse SRB measure if and only if the folding-type entropy formula holds and the Jacobian series is integrable. To overcome the degeneracy, we develop
Pesin theory for  general $C^{1+\alpha}$ maps. 
\end{abstract}	
}
\maketitle	

\section{Introduction}
	
Let $M$ be a compact $\mathsf{d}$-dimensional Riemannian manifold  without boundary, $f: M\to M$ be a  differentiable map, and   $\mu$ be an $f$-invariant Borel probability measure. The well-known Ruelle (or Margulis-Ruelle)  inequality \cite{Ruelle78} 
states that the metric entropy $h_{\mu}(f)$ is upper bounded by the sum of positive Lyapunov exponents 
\begin{equation}\label{ruelle ineq}
	h_{\mu}(f)\leq  \int\sum\nolimits_{\lambda_i(x)>0}\lambda_i(x)d\mu(x)\end{equation}
where, by Oseledets theorem \cite{Oseledets}, the Lyapunov exponents $\{\lambda_i(x): 1\le i\le \mathsf{d}\}$ 
exist on a full $\mu$-measure set.
The reverse inequality to \eqref{ruelle ineq}
was established by Pesin \cite{Pesin} 
under the assumption that $f$ is $C^{2}$ and  $\mu$ is equivalent to 
the Lebesgue measure induced by the Riemannian metric on $M$. This yields the equality 
\begin{equation}\label{pesin eq}h_{\mu}(f)= \int\sum\nolimits_{\lambda_i(x)>0}\lambda_i(x)d\mu(x),
\end{equation}
which is known as the Pesin entropy formula. 
 
The Pesin entropy formula  
has  been extensively discussed in various settings, including non-invertible cases (see \cite{Led81, Le-S, Liu98, LY1, LY2,Qian-Zhang95, LQ_book, Sun-Tian, CCE2015}, etc.). A 
major extension 
is 
the generalization 
of  $\mu$ 
to 
the Sinai-Ruelle-Bowen (SRB) measure \cite{Le-S}, an $f$-invariant measure that has absolutely continuous conditional measures on unstable manifolds, 
which, as shown by Ledrappier and Young \cite{LY1}, is the only class of measures for which 
the entropy formula \eqref{pesin eq} holds.  
The SRB measure 
was originally proposed as a generalization of the  
Gibbs distribution in the mathematical treatment of 
statistical physics
\cite{Sinai, Ruelle_measure, Bowen_book}. It  
is regarded as the measure that is most ``close to" the volume when $f$ is not volume-preserving.

In the 1990s, the ergodic theory of smooth dynamical systems was proposed  as a useful tool to investigate the non-equilibrium statistical mechanics, 
particularly at the  conceptual level (see \cite{Gallavotti_Cohen95,Ruelle96, Ruelle_SRB, Ruelle99}, etc.). 
In the study of  non-equilibrium  entropy production 
in the framework of non-invertible differentiable maps,   
Ruelle  \cite{Ruelle96} 
introduced the folding entropy
$F_{\mu}(f)=H_{\mu}(\epsilon|f^{-1}\epsilon),$    which captures the complexities 
when $f$ ``folds" different states onto a single one\footnote{The $\epsilon$ denotes the partition into single points; see \cite{Ruelle96, Liu03, LW2, Shu, Mihailescu, Wu_Zhu} for more discussions on the folding entropy.}, and proposed  the following inequality, referred to as the folding-type Ruelle inequality, 
relating the folding entropy and  negative Lyapunov exponents in the backward direction 
\begin{eqnarray} \label{ruelle ineq non-inv}
h_{\mu}(f)\leq F_{\mu}(f)-\displaystyle\int\sum\nolimits_{\lambda_i(x)<0}\lambda_i(x)d\mu(x).
\end{eqnarray}
The folding-type Ruelle inequality for  maps with degeneracy
 was  proved by Liu \cite{Liu03}  under a series of polynomial-type  conditions, 
which were later removed by the authors, thereby proving \eqref{ruelle ineq non-inv} for all $C^{1+\alpha}$ maps \cite{LW1}. 

The main concern of this paper is the equality in \eqref{ruelle ineq non-inv}, that is,   
\begin{eqnarray}\label{eq non-inv of folding type} 
h_{\mu}(f)= F_{\mu}(f)-\displaystyle\int\sum\nolimits_{\lambda_i(x)<0}\lambda_i(x)d\mu(x)
\end{eqnarray} 
and we call it 
the {\it entropy formula of folding type}.  As indicated in \cite{Ruelle96, Liu08}, \eqref{eq non-inv of folding type}
would reveal the vanishing  or strict positivity of the entropy production $e_{\mu}(f):=F_\mu(f)-\int\log|\det D_xf|d\mu$ in terms of an SRB measure $\mu$ (see Definition 1.2 of \cite{QZhu}  for the precise definition in the non-invertible setting via the inverse limit space).
To see this, note that  $\mu$ always satisfies (\ref{pesin eq}), and if the equality in \eqref{ruelle ineq non-inv} 
also holds, one can then deduce 
\begin{eqnarray*} e_\mu(f)&=&F_{\mu}(f)-\displaystyle\int\sum_{\lambda_i(x)<0} \lambda_i(x) d\mu(x)-\displaystyle\int\sum_{\lambda_i(x)>0} \lambda_i(x) d\mu(x)\\[2mm]
&=& F_{\mu}(f)-\displaystyle\int\sum_{\lambda_i(x)<0} \lambda_i(x) d\mu(x)-h_{\mu}(f)= 0.
\end{eqnarray*}

Concerning the inverse process,  
\begin{Def}
Let $\mu$ be an $f$-invariant measure. The measure $\mu$ is called an inverse SRB  measure  if its conditional measures  on the stable manifolds   are absolutely continuous.  
\end{Def}

As a natural analogy 
to the SRB measure for the Pesin entropy formula, 
it was conjectured in \cite{Liu03} that:

\medskip

\noindent{{\bf Conjecture}:
Let $\mu$ be an $f$-invariant measure.
Then $\mu$ is an  inverse SRB  measure 
	{\it if and only if} the entropy formula of folding type \eqref{eq non-inv of folding type} holds.
}

\medskip

The aim of this paper is
to establish the validity of the 
{\bf Conjecture} in the broad setting where $f$ is a $C^{1+\alpha}$ ($0<\alpha<1$) map  possibly with degeneracy, i.e., $\Sigma_f=\{x\in M: \det(D_xf)=0\}\neq\emptyset,$ and $\mu$ satisfies the integrability condition for the Jacobian:
\begin{eqnarray}\label{integrable}
\Big|\int\log |\det (D_xf)| d\mu(x)\Big|<\infty.
\end{eqnarray}
\medskip

For $C^2$ maps without degeneracy, Liu \cite{Liu08} proves the {\bf Conjecture}, where  the ``if" statement is established under the following  condition   (H)' on the 
the Jacobian $J_f(x)=\frac{d \mu\circ f}{ d\mu}(x)$. 

\bigskip

\noindent (H)' For $\mu$-a.e. $x,$ the Jacobian $J_f(y)$ is well-defined  
on $V^s(x)$, the arc-connected component of the stable manifold $W^s(x)$ containing $x$, such that
$\prod_{n=0}^\infty\dfrac{J_f(f^nx)}{J_f(f^ny)}$ converges and is bounded away from $0$ and $\infty$ on any given neighborhood of $x$ in $V^s(x)$ whose $d_{W^s}$-diameter (i.e., the distance along $V^s(x)$)
 is finite; moreover, 
$
\sum_{z:fz=y}\dfrac{1}{J_f(z)}=1  
$
holds almost everywhere on 
$V^s(x)$ with respect to  $\lambda^s_x$,  the  Lebesgue measure on $W^s(x)$.

\bigskip

Note that  if  $J_f(x)=\frac{d \mu\circ f}{ d\mu}(x)$ extends to  a H\"older continuous function  $J_f:M\to[1,+\infty)$ (the condition (H) in \cite{Liu08}), 
then (H)' is satisfied.  

\medskip

To deal with the degeneracy, this paper generalizes (H)' to the following (H)'', extending the boundedness condition of the Jacobian series to integrability.  It will be shown that    (H)'' is also  necessary  for inverse SRB measure, thus  establishing the equivalent  criterion  for inverse SRB measures.  The equivalence was previously unknown, even in the non-degeneracy setting. 
\medskip 

\noindent (H)'' There exists an increasing  measurable partition $\xi$ of $M$ subordinate to the $W^s$-manifolds (see Definition \ref{Def}) such that   for $\mu$-a.e. $x,$ the Jacobian $J_f$ and  a function $\psi$   are  $\lambda^s_x$-a.e.  well-defined  
on $\xi(x)$, and  
\begin{eqnarray}\label{integral_cond}
\int_{\xi(x)} \Theta(x,y) d\lambda^s_x(y)<\infty,
\end{eqnarray}
where  $\Theta(x,y)$  is   the  Jacobian series:
\[\Theta(x,y):=\lim_{n \rightarrow\infty}   \Big(\prod_{k=0}^n\dfrac{J_f(f^kx)}{J_f(f^ky)}\Big)\frac{\psi(f^{n+1}y)}{\psi(f^{n+1}x)}.\]
Moreover, 
for $\lambda^s_x$-a.e. $y\in \xi(x),$ it holds that
\begin{eqnarray}\label{sum=1}
\sum\nolimits_{z\notin\Sigma_f:fz=y}\dfrac{1}{J_f(z)}=1. 
\end{eqnarray}
\medskip

\begin{Rmk}\label{rem:generalization}
{\rm 	By taking $\psi\equiv 1$ in  condition (H)'', the condition \eqref{integral_cond} reduces to 		\begin{eqnarray*}\label{integral_condn}
			\int_{\xi(x)}   \prod_{k=0}^{\infty}\dfrac{J_f(f^kx)}{J_f(f^ky)} d\lambda^s_x(y)<\infty.
		\end{eqnarray*}
		Hence, condition (H)'' is  a generalization of   (H)'.
}
\end{Rmk}

The  main result reads  as follows. 
\begin{Thm}\label{thm1}
Let $M$ be a compact Riemannian manifold without boundary, $f:M\to M$ be a $C^{1+\alpha}$ map, and $\mu$ be an $f$-invariant Borel probability measure  satisfying the integrability condition \eqref{integrable} for the Jacobian. 
Then $\mu$ is an inverse SRB measure if and only  if the entropy formula of folding type 
\begin{eqnarray*}\
h_{\mu}(f)= F_{\mu}(f)-\displaystyle\int\sum\nolimits_{\lambda_i(x)<0}\lambda_i(x)d\mu(x)
\end{eqnarray*} 
and condition (H)'' hold. 		
\end{Thm}
\begin{Rmk}\label{Rem1}
{\rm 
If  an SRB measure $\mu$ is also  inverse SRB (i.e., $\mu$ has absolutely continuous conditional measures with respect to both the stable and unstable foliations), and satisfies the integrability condition \eqref{integrable}, 
then Theorem \ref{thm1} yields that its entropy production $e_{\mu}(f)=0.$
}
\end{Rmk}

As applications of Theorem \ref{thm1}, Section \ref{app} establishes the entropy formula of folding type \eqref{eq non-inv of folding type} for all conservative maps, using the “only if” part of  Theorem \ref{thm1}. It also proves the existence of inverse SRB measures for a class of dissipative degenerate systems, based on the “if” part of Theorem \ref{thm1}. 

With the establishment of Theorem \ref{thm1}, this paper makes two main contributions: 
(i) it develops  Pesin theory for all $C^{1+\alpha}$ maps, including those with degeneracy, and establishes the entropy formula of folding type for general inverse SRB measures; 
(ii) it generalizes Liu's boundedness condition (H)' in \cite{Liu08} to a weaker integrability condition (H)", thereby  providing a full characterization of inverse SRB measures. In particular, as a condition to guarantee the existence of inverse SRB measure, (H)'  is not necessary; this gap is addressed  by the new condition (H)''.

We note that differentiable maps with unbounded derivatives, such as maps with singularities, 
appear in various settings
(see, e.g., \cite{Katok-Strelcyn_book, Le-S, Lima}), and our method exhibits the possibility for such systems (see, for instance, the controllable setting in  \cite{Katok-Strelcyn_book} and the integrable setting  in \cite{LiaoQiu}).  In this paper, the degeneracy 
leads to the loss of control over the local behaviors of $f$ 
in the vicinity of $\Sigma_f.$ The  
integrability  property
\eqref{integrable}  
plays a crucial role in this regard,  
by  controlling the decay  of the determinant of $f$
along trajectories which may approach 
$\Sigma_f$ as time evolves. 
To overcome the degeneracy, the Pesin theory for general $C^{1+\alpha}$  maps is developed. 
The development of Pesin theory plays an important role in various settings of applications. For the adaptation to symbolic dynamics, it has also been developed for one dimensional maps with non-uniform expansion \cite{Lima1} and  for non-uniformly hyperbolic maps with singularities in higher dimensions \cite{Lima_book}.

This paper is organized as follows. In Section \ref{pes}, results in  Pesin theory, including the Pesin sets, the local stable manifolds and the associated measurable partitions, are  established for differentiable $C^{1+\alpha}$ maps with degeneracy under the integrability condition on the invariant measure. The  necessity and sufficiency  statements  in Theorem \ref{thm1}
are  
proved in Sections \ref{nes} and \ref{suf}, respectively.  In Section \ref{app}, applications in  conservative  and  dissipative systems are presented. 

\section*{Acknowledgments}
The authors are  grateful to the referees for their valuable suggestions, which have significantly improved both the  results and the proofs, particularly the recommendations regarding the  applications in Section \ref{app}.	

\section{Pesin theory for general $C^{1+\alpha}$ maps}\label{pes}
\subsection{Lyapunov neighborhoods for $C^{1+\alpha}$ maps with degeneracy}
In this section, we establish the Lyapunov metric and Lyapunov neighborhoods, 
which generalize 
the classical Pesin theory for  non-uniformly hyperbolic diffeomorphisms to  
the setting of $C^{1+\alpha}$ maps with degeneracy. 

First, recall  Oseledets theorem \cite{Oseledets},  which guarantees the existence of Lyapunov exponents. 

\begin{Thm}\label{thm:MET} There exists a measurable set $\Delta_0\subseteq M$ with $\mu(\Delta_0)=1$ such that for each  $x\in\Delta_0,$ there exists a  splitting   $T_x M=F_1(x) \oplus  F_2(x)\oplus \cdots \oplus F_{q(x)}(x)$ and a sequence of real numbers
$-\infty<\lambda_1(x)<\cdots<\lambda_{q(x)}(x)<+\infty$ such that
\begin{eqnarray*}
&&\lim\limits_{n\to +\infty}\dfrac{1}{n}\log\|D_xf^nv\|=\lambda_i(x), \quad \ \forall v\in F_i(x),\,\,1\le i\le q(x),\\[1mm]
&&\lim\limits_{n\to +\infty}\dfrac{1}{n} \log \angle (D_xf^nF_{i_1,\cdots,i_j}(x),  D_xf^nF_{l_1,\cdots,l_t}(x))=0,
\end{eqnarray*}
where  the first convergence is uniform in each $F_i(x)$, and $F_{i_1,\cdots,i_j}(x)$
and $F_{l_1,\cdots,l_t}(x)$
denote direct sums of pairwise distinct subspaces 
$F_{i_1}(x)$, $\cdots$, $F_{i_j}(x)$ and 
$F_{l_1}(x)$, $\cdots$, $F_{l_t}(x)$, respectively. In particular, if $\mu$ is ergodic, then $q:=q(x),$ $\lambda_i:=\lambda_i(x)$ for $1\le i\le q$ are constants  for $\mu$-almost everywhere. 
\end{Thm}
\begin{proof} 
In the classical Oseledets theorem, the splitting form shown above is stated for invertible systems. Since  this paper deals  with non-invertible maps, we pass to the inverse limit space $\bar M$ of $(M,f),$ defined by \[\bar M=\{\bar x=(\cdots, x_{-1},x_0,x_1,\cdots):x_i\in M,fx_i=x_{i+1}, i\in\mathbb Z\},\] which consists of all full trajectories of $f$ on $M.$ Let $\tau: \bar M\to \bar M$ be the left-shift and 
$\pi:\bar M\to M$, $\bar x\mapsto x_0$ be the natural projection. There exists a unique $\tau$-invariant measure $\bar\mu$ on $\bar M$ such that $\bar\mu=\mu\circ\pi^{-1}$. Note that the derivative $D_xf$ can be viewed as 
a linear operator  $\mathbb{R}^{\mathsf{d}} \to \mathbb{R}^{\mathsf{d}}$, and thus defines a measurable cocycle $A(\bar x) =D_{x_0}f $ over the system $(\bar M,  \tau, \bar\mu)$.  By applying  the Oseledets theorem to this cocycle,  there exists a measurable $\tau$-invariant set $\bar\Delta_0\subseteq\bar M$ of full $\bar\mu$-measure such that for each $\bar x\in\bar\Delta_0$, there exists a splitting  \[\{\bar{x}\} \times \mathbb{R}^{\mathsf{d}} = \{\bar{x}\} \times ( F_1(\bar{x}) \oplus  F_2(\bar{x})\oplus \cdots \oplus F_{q(\bar{x})}(\bar{x}))\] and a sequence of real numbers
$-\infty<\lambda_1(\bar{x})<\cdots<\lambda_{q(\bar{x})}(\bar{x})<+\infty$ satisfying the statement in the theorem. Now, for $\mu$-a.e. $x_0\in M,$ choose $\bar{x}\in\bar\Delta_0$ such that $\pi(\bar x)=x_0$. Define the  splitting at $x_0$ by \[\{x_0\}   \times \mathbb{R}^{\mathsf{d}}= \{x_0\} \times ( F_1(\bar{x}) \oplus  F_2(\bar{x})\oplus \cdots \oplus F_{q(\bar{x})}(\bar{x}))\]  and 	set $\lambda_i(x_0):=\lambda_i(\bar{x})$. Then the desired conclusion holds for the non-invertible system $(M, f, \mu)$. 
\end{proof}

Let
$I=\{x\in\Delta_0: 
\lambda_1(x)\ge 0\}.$ 
Recall that $\Sigma_f=\{x\in M: \det(D_xf)=0\}$ denotes the set of degenerate points of $f$.
From \eqref{integrable},  it follows that $\mu(\Sigma_f)=0.$ Let
$$\Delta_1=\Delta_0\backslash(\cup_{n\geq0}f^{-n}\Sigma_f\cup I).$$ 
 Note that $f\Delta_1\subseteq\Delta_1.$ 
For any $x\in\Delta_1$  and  $n\ge0,$ since  $f^ix\notin \Sigma_f$ for $0\le i<n$, $D_xf^n$ is invertible, and thus $(D_xf^n)^{-1}$ is well-defined.

For $x\in\Delta_1,$ define 
\begin{eqnarray*}
 s(x)=\max\{1\leq i\leq q(x): \lambda_i(x)<0\},\quad E^s(x)
=\bigoplus_{i=1}^{s(x)}F_i(x).   
\end{eqnarray*}
Let 
\begin{eqnarray*}
& &E_0(x)=E^s(x),\quad H_0(x)=E_0(x)^{\perp},
\end{eqnarray*}
and for $m\geq1,$ denote
\begin{eqnarray*}
& &E_m(x)=(D_xf^m)E_0(x),\ \ H_m(x)=(D_xf^m)H_0(x).
\end{eqnarray*}
Note that $E_m(x), H_m(x)\subset T_{f^mx}M$ for all $m\ge0.$
For any  two real numbers $a<b,$ define the set 
\begin{eqnarray*}
    \Lambda_{a,b}=\{x\in\Delta_1: \lambda_{s(x)}<a, \lambda_{s(x)+1}>b\}.
\end{eqnarray*}
Throughout the paper, let $\varepsilon>0$ denote a sufficiently small constant. 

\medskip

For a linear map $P: (M_1, \|\cdot\|_1)\to (M_2, \|\cdot\|_2)$, its co-norm is defined by \[\sigma(P)=\inf_{0\neq v \in M_1} {\|Pv\|_2}/{\|v\|_1}.\] 
The following Lemma \ref{Pesin-set} provides the construction of Pesin sets 
for general $C^{1+\alpha}$ differentiable maps $f$, including those admitting degeneracy.

\begin{Lem}
[Pesin sets in degenerate setting]
\label{Pesin-set}
There exists a measurable subset $\widetilde\Lambda_{a,b}\supseteq\Lambda_{a,b}$ such that $\widetilde\Lambda_{a,b}=\cup_{k\in\mathbb Z^+}\widetilde\Lambda_{a,b,k},$ where each $\widetilde\Lambda_{a,b,k}$ consists of points  $x$ whose tangent space admits a splitting $T_xM=E_0(x) \oplus H_0(x)$ satisfying the following properties: 

\begin{itemize}
\item[(i)] $\|Df^n\mid_{E_m(x)}\| \leq e^{k\varepsilon}e^{(a+\varepsilon) n}e^{m\varepsilon},\quad \forall  \ m,n\geq0;$ \\

\item[(ii)] $\sigma(Df^n\mid_{H_m(x)}) \geq e^{-k\varepsilon}e^{(b-\varepsilon)n}e^{-m\varepsilon},\quad \forall  \  m,n\geq0;$ \\
\item[(iii)]  $\angle(E_m(x),  H_m(x))\geq e^{-k\varepsilon}e^{-m\varepsilon}, \ \forall m\geq0;$\\

\item[(iv)] $\sigma(D_{f^mx}f)\geq e^{-k\varepsilon}e^{-m\varepsilon},\ \forall m\geq0$.
\end{itemize}
\end{Lem}

We note that items (i)-(iii) above correspond to  the classical construction of Pesin sets in the non-uniform hyperbolic theory for $C^{1+\alpha}$ diffeomorphism (see, for instance, section 4.1 in \cite{Pollic}); item (iv) serves to control the decay of co-norms of $Df$ along trajectories, addressing the possible degeneracy of $f$.    

\begin{proof}[Proof of Lemma \ref{Pesin-set}]
We first prove  (iv). 
Since 
$ |\det (D_xf)| \le\sigma(D_xf)\cdot\|D_xf\|^{\mathsf{d-1}}$
and $L:=\sup_{x\in M}\|D_xf\|<\infty$, it follows from \eqref{integrable} that
$$\displaystyle\int\log(1/\sigma(D_xf))d\mu(x)\le -\displaystyle\int\log|\det(D_xf)|d\mu(x)+(\mathsf{d}-1)\int \log L<\infty.$$ 
Applying the Birkhoff Ergodic Theorem,
there exists a measurable set $\Delta^\prime\subseteq M$ with 
$f\Delta^\prime\subseteq\Delta^\prime$ and  $\mu(\Delta^\prime)=1$ such that for any $x\in\Delta^\prime,$
\begin{eqnarray}\label{Delta'}
\lim\limits_{n\to+\infty}\dfrac{1}{n}\log\sigma(D_{f^nx}f)=0.
\end{eqnarray}
For any $x\in\Delta^\prime,$  $\vep'>0$, let $B(x)>0$ be the largest constant satisfying
\begin{eqnarray*}
\sigma(D_{f^nx}f)\geq B(x)e^{-\varepsilon' n},\,\,\forall\,n\ge 0.    \end{eqnarray*}
By choosing a sufficiently large $k$ such that $B(x)\geq e^{-k\varepsilon'}$, (iv) is obtained. 

The (iii) follows similarly. By Theorem \ref{thm:MET}, 
\begin{eqnarray*}\lim_{n\to +\infty}\frac {1}{n} \log \angle(Df^n(F_{i_1,\cdots,i_j}(x)),  Df^n(F_{l_1,\cdots,l_t}(x)))=0,
\end{eqnarray*} 
where $F_{i_1}(x)$, $\cdots$, $F_{i_j}(x)$, $F_{l_1}(x)$, $\cdots$, $F_{l_t}(x)$ are pairwise distinct. Hence, in particular, 
\begin{eqnarray*}
\lim_{n\to +\infty}\frac {1}{n} \log\angle(  Df^nE_0(x),  Df^nH_0(x))=0.
\end{eqnarray*}
By choosing $k$ sufficiently large, it then follows that
\begin{eqnarray*}
\angle(Df^mE_0(x),  Df^mH_0(x))\geq e^{-k\varepsilon'}e^{-m\varepsilon'}, \quad \forall m\geq0.
\end{eqnarray*}
Hence, (iii) follows from the construction of $E_m$ and $H_m$.

To obtain (i) and (ii), note that  for each $1\le i\le q(x)$ and $m\ge 0$, 
	~$$  \underset{n\rightarrow+\infty}{\lim}\frac{1}{n}\log \sigma(Df^n\mid_{Df^m(F_i(x))})= \underset{n\rightarrow+\infty}{\lim}\frac{1}{n}\log\|Df^n\mid_{Df^m(F_i(x))}\|=\lambda_i(x).$$ 
 Thus, for any $\varepsilon'>0,$ there exists $N(x,m)$ such that for all $n\geq\,N(x,m)$,  
	$$ e^{(\lambda_i-\varepsilon')n }\le \sigma(Df^n\mid_{Df^m(F_i(x))})   \le  \|Df^n\mid_{Df^m(F_i(x))}\|  
	\leq\,e^{(\lambda_i+\varepsilon')n}.$$
Define $C(x,t)$ as the minimal positive constant such that for all $n\ge0,$ 
	~$$\frac{1}{C(x,m)} e^{(\lambda_i-\varepsilon')n}  \le   \sigma(Df^n\mid_{Df^mF_i(x)})   \le  \|Df^n\mid_{Df^mF_i(x)}\|    \le C(x,m)  e^{(\lambda_i+\varepsilon')n}.$$
For any $m, n >0$, on one hand, 
\begin{eqnarray*}
  \|Df^n\mid_{Df^mF_i(x)}\|   &\le  & 	\frac{ \|Df^{n+m}\mid_{F_i(x)}\|}{\sigma(Df^m\mid_{F_i(x)})}\\[2mm]
	&\le & \frac{  C(x,0)  e^{(\lambda_i+\varepsilon')(n+m)}} {\frac{1}{C(x,0)} e^{(\lambda_i-\vep')m}}
	=  C^2(x,0)    e^{(\lambda_i+\varepsilon')n} e^{2m\vep};
\end{eqnarray*}
on the other hand, 
\begin{eqnarray*}
 \sigma(Df^n\mid_{Df^mF_i(x)}) &\ge  & \frac{ \sigma(Df^{n+m}\mid_{F_i(x)})}{\|Df^m\mid_{F_i(x)}\|}\\[2mm]
	&\ge & \frac{\frac{1}{C(x,0)}   e^{(\lambda_i-\varepsilon')(n+m)}} {C(x,0) e^{(\lambda_i+\vep')m}}
	= \frac{ 1}{C^{2}(x,0) }   e^{(\lambda_i-\varepsilon')n} e^{-2m\vep'}.
\end{eqnarray*}
Therefore,  $$C(x,m)\leq C^2(x,0)e^{2m\varepsilon'}.$$
 By choosing $k$ large enough so that   $C^2(x,0)\le e^{k\vep'}$, one obtains that for $x\in \Lambda_{a,b}$ and $1\le i\le s(x)$, $$\|Df^n\mid_{Df^mF_i(x)}\|\leq e^{k\varepsilon'}e^{(a+\varepsilon') n}e^{2m\varepsilon'},\quad\forall  m,n\geq0.$$

For any unit vector $v\in E_m(x)$, it can be represented as \[v=\sum\nolimits_{1\le i\le s(x)}v_i=v_i+v_i^c,\]  where $v_i\in Df^mF_i(x)$ and $v_i^c\in Df^mF_i^c(x):=\bigoplus_{1\le j\le s(x),\,j\neq i} Df^mF_j(x).$  By the Sine Law (see Figure \ref{fig:vector} with $\theta=\pi-\angle(v_i,v_i^c)$, $\beta=\angle(v,v_i^c)$), we have   \begin{align*}
\dfrac{\|v\|}{\sin\angle(v_i,v_i^c)}=\dfrac{\|v_i\|}{\sin\angle(v,v_i^c)}.
\end{align*}
Since $\sin\angle(v_i,v_i^c)=\sin\angle(Df^mF_i(x), Df^mF^c_{i}(x))\geq\sin(e^{-k\varepsilon'}e^{-m\varepsilon'}),$ 
we have  
\begin{eqnarray*}
\|v_i\|\le\dfrac{1}{\sin\angle(v_i,v_i^c)}\leq \frac{1}{\sin (e^{-k\varepsilon'}e^{-m\varepsilon'})}
\end{eqnarray*}
for all $1\le i\le s(x).$ Hence, 
\begin{eqnarray*}
\|Df^nv\| \leq  \mathsf{d}\max_{1\le i\le s(x)} \|Df^nv_i\|\le  \mathsf{d} \frac{e^{k\varepsilon'}e^{(a+\varepsilon') n}e^{2m\varepsilon'}}{\sin (e^{-k\varepsilon'}e^{-m\varepsilon'})},\quad\forall n\ge0.
\end{eqnarray*}
By the arbitrariness of $v$, it follows that 
\[\|Df^n\mid_{E_m(x)}\|\leq c_0e^{2k\varepsilon'}e^{(a+\varepsilon') n} e^{3m\varepsilon'},\quad\forall\ m,n\geq0,\]
where the constant $c_0=\mathsf{d}\sup_{0< \vartheta\le \frac{\pi}{2}} \frac{\vartheta}{\sin \vartheta}$. By taking  $\vep'$ sufficiently small  and   $k$ sufficiently large, (i) is obtained with $\vep=4\vep'$ and $c_0\le e^{k\vep'}$. Item
(ii) then follows by analogous arguments. 
\end{proof}

\begin{figure}
 \centering    \includegraphics[width=0.6\linewidth]{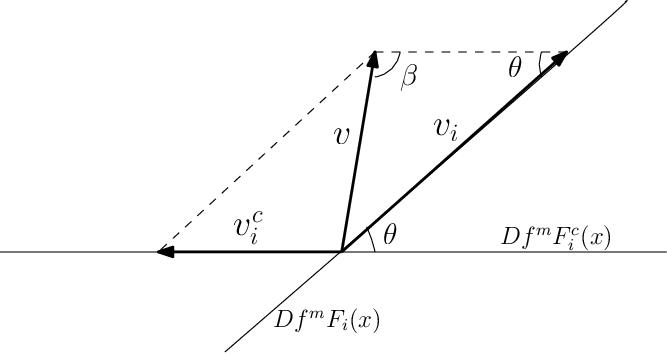}   \caption{Geometric decomposition of  unit vector $v \in E_m(x)$ into $v_i \in Df^m F_i(x)$ and $v_i^c \in Df^m F_i^c(x)$.}   \label{fig:vector}
\end{figure}

It is not hard to see that  $\widetilde\Lambda_{a,b,k}$ is compact. Moreover,  the following holds. 
\begin{Prop}
For each $m\ge0,$ the subspace $E_m(x)$ is uniquely determined and depends continuously on $x\in \widetilde\Lambda_{a,b,k}$.
\end{Prop}  
\begin{proof}  It suffices to  consider $E_0(x)$, as  $E_m(x)=Df^m(E_0(x))$.   

For the uniqueness, suppose that for $ x\in \widetilde\Lambda_{a,b,k}$, there is another splitting $T_xM=G(x)\oplus Z(x)$ such that 
\begin{eqnarray*}
&&\lim\limits_{n\to +\infty}\dfrac{1}{n}\log\|D_xf^nv\|<a, \quad \ \forall v\in G(x),\\
&&\lim\limits_{n\to +\infty}\dfrac{1}{n}\log\|D_xf^nv\|>b, \quad \ \forall v\in Z(x).
\end{eqnarray*} 
Since the splitting $T_xM=E_0(x)\oplus H_0(x)$ also satisfies the above estimates, any non-zero $v\in G(x)$ must have zero projection on $H_0(x)$ which implies $G(x)\subseteq E_0(x)$.  The reverse inclusion   $E_0(x)\subseteq G(x)$ follows  similarly. Hence, $G(x)=E_0(x)$, and the uniqueness is established.    

For the continuity, let $\{x_p\}\subset  \widetilde\Lambda_{a,b,k}$ be a sequence converging to some $y\in\widetilde\Lambda_{a,b,k}$ as $p\to\infty$. Then the sequence $\{E_0(x_p)\}$ has a convergent subsequence with limit $\tilde E$, which satisfies the properties in Lemma \ref{Pesin-set}. The uniqueness then implies that $\tilde E=E_0(y)$, thereby establishing the continuity of $E_0$. 
\end{proof}

For $x\in\widetilde\Lambda_{a,b}$ and any $m\ge 0$, let   
\begin{eqnarray*}
\langle v_1,v_2\rangle_{(x,m)}^\prime=
\begin{cases}   \sum_{n=0}^{+\infty}e^{-2(a+2\varepsilon)n}\langle D_{f^mx}f^nv_1,D_{f^mx}f^nv_2\rangle, \quad\forall v_1, v_2\in E_m(x),\\\\
   \sum_{n=0}^{+\infty} e^{2(b-2\varepsilon)n}\langle (D_{f^{m-n}x}f^n)^{-1}v_1,(D_{f^{m-n}x}f^n)^{-1}v_2\rangle,\quad\forall v_1, v_2\in H_m(x),\\\\
   0,\quad\forall v_1\in E_m(x), v_2\in H_m(x),
\end{cases}
\end{eqnarray*}
where $\langle\ ,\ \rangle$ is the inner product induced by the Riemannian metric on 
$T_{f^mx}M.$ Note that the second equation is well-defined for  $x\in\Delta_1.$ 
  
The  Lyapunov metric   $\|\cdot\|_{(x,m)}^\prime $ is defined as the norm induced by $\langle\cdot,\cdot\rangle_{(x,m)}^\prime$. By definition,  under the Lyapunov metric,  the restrictions 
$Df|_{E_{m}(x)}$ and   $Df|_{H_{m}(x)}$ illustrate uniform contracting and expanding properties, respectively. This leads to the following normalization with respect to the Lyapunov metric. 

\begin{Prop}\label{prop:new_norm}
For $x\in\widetilde\Lambda_{a,b,k},$ the sequence of norms $\{\|\cdot\|^\prime_{(x,m)}: m\geq0\}$ satisfies the following properties:  
\begin{itemize}
\item[(i)] $\|D_{f^mx}fv\|^\prime_{(x,m+1)}\leq e^{(a+2\varepsilon)}\|v\|^\prime_{(x,m)},\quad\forall v\in E_m(x);$\\

\item[(ii)]$\|D_{f^mx}fv\|^\prime_{(x,m+1)}\geq e^{(b-2\varepsilon)}\|v\|^\prime_{(x,m)},\quad \forall v\in H_m(x);$\\

\item[(iii)]$\frac{1}{2}\|v\|\leq\|v\|^\prime_{(x,m)}\leq C_0' e^{k\varepsilon}e^{m\varepsilon }\|v\|,\quad \forall v\in T_{f^mx}M,$
\end{itemize}
where the constant $C_0'$ in (iii) is independent of $k.$
\end{Prop}

\begin{proof}

	For  $v\in E_m(x)$,
		\begin{eqnarray*}
		{\|D f(v) \|'} ^2 
		&=& \sum_{n=0}^{\infty} e^{-2(a+2\varepsilon)n}  \|D f^{n+1}(v) \|^2 \\
		&=& e^{2(a+2\varepsilon)}  \sum_{n=0}^{\infty} e^{-2(a+2\varepsilon)(n+1)}  \|D f^{n+1}(v) \|^2\\		&\le&e^{2(a+2\varepsilon)}  \sum_{n=0}^{\infty} e^{-2(a+2\varepsilon)n}  \|D f^n(v) \|^2
	\ \le\ e^{2(a+2\varepsilon)}  \| v \|'^2.
		\end{eqnarray*}
Thus, (i) holds. A similar argument yields (ii).    

To prove (iii), we first estimate the upper bound. For $v\in E_m(x)$, we have 
\begin{eqnarray*}
{\|v\|'} ^2 
&=& \sum_{n=0}^{\infty} e^{-2(a+2\varepsilon)n}  \|D f^{n}(v) \|^2 \\
&\le &  \sum_{n=0}^{\infty} e^{-2(a+2\varepsilon)n}  e^{2(k+m)\vep} e^{2(a+\varepsilon)n}  \|v\|^2 \\
&=&   e^{2(k+m)\vep}  \sum_{n=0}^{\infty} e^{-2\varepsilon n}  \|v\|^2.
\end{eqnarray*}
Denote ${C_0'}^2:=\sum_{n=0}^{\infty} e^{-2\varepsilon n} $. Then $\| v \|' 
\le    e^{(k+m)\vep}  C_0' \|v\|,$
which, by an analogous estimate, also holds for 
$v\in H_m(x)$. 
Thus, 
\begin{eqnarray*}\label{v}
\| v \|' 
\le    e^{(k+m)\vep}  C_0' \|v\|,\quad\forall v\in T_{f^mx} M.
\end{eqnarray*}

To obtain the lower bound, for   $v\in T_{f^mx} M$, write $v=v_E+v_H$ with $v_E\in E_m$ and $v_{H}\in H_m$. Then  
\begin{eqnarray*}
	{\| v \|}^2 &= &\langle v_E + v_H,  v_E + v_H \rangle \\[2mm]
	&= &{\| v_E \|}^2 + {\| v_H \|}^2 + 2\langle v_E, v_H\rangle\\[2mm]
		&\le &2{\| v_E \|}^2 + 2{\| v_H \|}^2 
	\le  2{\| v_E \|'}^2 + 2{\| v_H \|'}^2= 2 { \| v \|}'^2.
\end{eqnarray*}
Hence, 
${\| v \|}'\ge\frac{1}{2}\|v\|.$
\end{proof}

According to (iii) of  Proposition \ref{prop:new_norm},  for $x\in\widetilde\Lambda_{a,b,k}$ and  $m\ge0$, it holds that 
\begin{eqnarray*}
\|D_{f^m(x)}f\|'&=&\sup_{v\neq 0} \frac{\|D_{f^m(x)}f(v)\|'}{\|v\|'}\nonumber\\
&\le&2  C_0' e^{k\varepsilon}e^{\varepsilon m}  \|D_{f^m(x)}f \|\ \le\ 2  C_0' e^{k\varepsilon}e^{\varepsilon m} L,
\end{eqnarray*}
where  $L=\max_{z\in M} \|D_zf\|$. 

By applying the exponential map at $x$, one may assume without loss
of generality that the analysis is conducted locally in $\mathbb{R}^{\mathsf{d}}$.  
For each point of Pesin set, the Lyapunov metric 
can be ``translated" to a small neighborhood on the tangent space. More specifically, for any  
$x\in \widetilde\Lambda_{a,b}$, $m\ge 0$, and a sufficiently small neighborhood $U$ of $f^{m}x$, the
tangent bundle over $U$ 
can be trivialized
by identifying $T_UM\equiv U\times
\mathbb{R}^{\mathsf{d}}$. Under this identification, for each $y\in U$, any  vector $v\in T_yM$  can be translated  to the corresponding
vector $\bar{v}\in T_{f^{m}x}M$.
A new
norm $\|\cdot\|''$ is defined by setting  $\|v\|''_y=\|\bar{v}\|'_{(x,m)}$, which agrees with $\|\cdot\|'_{(x,m)}$ on $T_{f^mx}M$. Let $d''$ denote the distance induced by $\|\cdot\|''$. 
By the H\"older continuity of $Df$, for any  $x\in\tilde\Lambda_{a,b,k}$ and any $y$ in a sufficiently small neighborhood of $x,$ one has
\begin{eqnarray*}
\|D_xf-D_yf\|''&\le& 2C_0'e^{k\vep}\|D_xf-D_yf\|\\
&\le&2C_0'e^{k\vep}C_0d^{\alpha}(x,y)\ \le\ C_0''e^{k\vep}d''^{\alpha}(x,y),  \end{eqnarray*}
where $C_0''=C_0'C_0 2^{\alpha+1}.$
Therefore, if  $d(x,y)\le e^{-\frac{3k\vep}{\alpha}}$,  
then 
\begin{eqnarray*}
 d''(x,y)\le C_0'e^{k\vep}d(x,y)\le C_0'e^{-\frac{2k\vep}{\alpha}},  
\end{eqnarray*}
and hence 
\begin{eqnarray}\label{holder_2norm}
\|D_xf-D_yf\|''&\le&  C_0''e^{k\vep}d''^{\alpha}(x,y)\nonumber\\
&=&C_0''e^{k\vep}(d''(x,y))^{\alpha/2}(d''(x,y))^{\alpha/2}\nonumber\\[2mm]&\le& C_0''( C_0')^{\alpha/2}(d''(x,y))^{\alpha/2}:=\tilde{C}_0(d''(x,y))^{\alpha/2}.
\end{eqnarray}
Therefore, the H\"older continuity remains valid under the norm  $\|\cdot\|''$, with H\"older exponent $\alpha/2$ and a uniform prefactor $\tilde C_0$ inside small neighborhoods. 
\medskip

To characterize the local behaviors of $f$ near regular points (i.e., the points in a suitable set   $\tilde\Lambda_{a,b}$), a standard method is to ``project" the Lyapunov metric from the tangent space onto a small neighborhood of the point on the manifold $M$.  
Such neighborhoods, with size 
decreasing at a rate not exceeding $e^{-\varepsilon/\alpha}$
along the orbits, are often referred to as Lyapunov neighborhoods in literature.

As this paper mainly concerns differentiable maps with degeneracy,  additional considerations are  required to build such
Lyapunov neighborhoods. 

\begin{Lem}\label{lem:non-degenerate-neighborhood}
There exists $r_k^\prime=r_0^\prime e^{-k\varepsilon/\alpha}>0,$ where $0<r_0^\prime<1$ is a constant independent of $k$, such that for any $x\in\widetilde\Lambda_{a,b,k}$ and $m\geq0$, 
$$\sigma(D_yf)\geq\dfrac{1}{2}e^{-k\varepsilon}e^{-m\varepsilon}>0,\quad\forall y\in B(f^mx,r_k^\prime e^{-m\varepsilon/\alpha}).$$
Consequently,  $f|_{B(f^mx,r_k^\prime e^{-m\varepsilon /\alpha})}$ is a diffeomorphism such that  for any $\rho\in (0,1)$, 
\begin{eqnarray*}
f\big(B(f^mx,  \rho r_k^\prime   e^{-m\varepsilon/\alpha})\big)\supseteq  B( f^{m+1}x, \frac{1}{2}\rho e^{-k\vep}r_k'e^{-2m\vep/\alpha}).
\end{eqnarray*}
\end{Lem}
\begin{proof}
By H\"{o}lder continuity of $Df,$ there exists $K>1$  such that for any $x,y$ close enough, 
\begin{eqnarray}\label{holder-ineq}
\|D_xf-D_yf\|\leq Kd(x,y)^\alpha.
\end{eqnarray}
Thus, for any $y\in M$ satisfying $d(y,f^mx)\leq(\dfrac{1}{2K}e^{-k\varepsilon}e^{-m\varepsilon})^{1/\alpha},$ it holds that
\begin{eqnarray*}
\sigma(D_yf)&\geq&\sigma(D_{f^mx}f)-Kd(f^mx,y)^\alpha\ \ \ \ \ {\text {by~(\ref{holder-ineq})}}\\
       &\geq&e^{-k\varepsilon}e^{-m\varepsilon}-Kd(f^mx,y)^{\alpha}\ \ \ \ \ \ \ {\text {by Lemma~\ref{Pesin-set}(iv)}}\\
       &\geq&\dfrac{1}{2}e^{-k\varepsilon}e^{-m\varepsilon}.
\end{eqnarray*}
The lemma is concluded by letting $r_0'=({1}/{2K})^{1/\alpha}$.
\end{proof}

\subsection{Measurable partitions subordinate to  stable manifolds}
In the study of dynamical behaviors 
of $C^{1+\alpha}$ diffeomorphism 
in the non-uniformly hyperbolic theory 
a fundamental result  
is the stable manifold theorem \cite{Pesin, Ruelle79}, which is typically established using the method  of graph transformation (see, for instance, \textsection 2 of \cite{HPS}). In the present setting, however, the map $f$ may be non-invertible and thus differs from a diffeomorphism. To address this, we apply the graph transformation technique to the inverse branches of $f$ in the backward directions.

The following lemma quantifies the Lyapunov neighborhoods  on which $f^{-1}$ behaves uniformly under the norm $\|\cdot\|''$. 

\begin{Lem}\label{prop:f-under-local-chart}
For any $x\in\widetilde\Lambda_{a,b,k}$ and any $m\geq0,$
denote $A_{x,m}= (D_{f^{m}x}f)^{-1}.$
Then
\begin{eqnarray*}
&&	\| A_{x,m}v\|''\geq e^{-a-2\vep}\|v\|'',\quad \forall v\in E_m(x),\\
&&	\| A_{x,m} v\|'' \leq e^{-b+2\vep}\|v\|'', \quad \forall v\in H_m(x).
	\end{eqnarray*}
Moreover, for any small $\delta>0$, there exists $\tilde{r}_k=\tilde{r}_0 e^{-30k\varepsilon/\alpha}>0,$ where $0<\tilde{r}_0<1$  depends on $\delta$ but  independent of $x$ and $k$, such that, if we denote
$\tilde f_{(x,m)}=(f\mid_{B(f^mx, \tilde r_ke^{-m\varepsilon})})^{-1}$, it holds that for any  $y\in   B( f^{m+1}x, \tilde{r}_k e^{-30(m+1)\vep/\alpha})$, 
\begin{eqnarray*}
	\| D_y\tilde f_{(x,m)} - A_{x,m}\|''<\delta.
\end{eqnarray*}
\end{Lem}
\begin{proof}
The first two inequalities directly follows from  Proposition \ref{prop:new_norm} (i)-(ii). It remains  to prove the last inequality. 
Denote  
 $\tilde{y}=\tilde f_{(x,m)} (y)$ for $y\in B( f^{m+1}x, \tilde{r}_k e^{-30(m+1)\vep/\alpha}).$ Then 
\begin{eqnarray*}
&&	\| D_y\tilde f_{(x,m)} - A_{x,m}\|\\[2mm]
&=&  	\| D_{y}(f\mid_{B(f^mx, \tilde r_ke^{-m\varepsilon})})^{-1}\circ (D_{\tilde{y}}f- D_{f^mx}f)\circ D_{f^{m+1}x}(f\mid_{B(f^mx, \tilde r_ke^{-m\varepsilon})})^{-1}\|\\[2mm]
	&\le &	\| D_{y}(f\mid_{B(f^mx, \tilde r_ke^{-m\varepsilon})})^{-1}\| \cdot \| D_{\tilde{y}}f- D_{f^mx}f\| \cdot\| D_{f^{m+1}x}(f\mid_{B(f^mx, \tilde r_ke^{-m\varepsilon})})^{-1}\|\\[2mm]	
	&= &  \frac{1}{\sigma(D_{\tilde{y}}f)}\cdot  \| D_{\tilde{y}}f- D_{f^mx}f\| \cdot \frac{1}{\sigma(D_{f^{m}x})}\\[2mm]
	&\le & 2 e^{2k\varepsilon}e^{2m\varepsilon} K d(\tilde{y}, f^mx)^{\alpha},
\end{eqnarray*}
where the last inequality is by Lemma \ref{lem:non-degenerate-neighborhood}. It  then follows from Proposition \ref{prop:new_norm} (iii) that 
\begin{eqnarray*}
	\| D_y\tilde f_{(x,m)} - A_{x,m}\|''
	\le     2C_0e^{k\varepsilon}e^{\varepsilon m} e^{2k\varepsilon}e^{2m\varepsilon}\cdot K d(\tilde{y}, f^mx)^{\alpha}.
\end{eqnarray*}

Take $ \tilde{r}_0=\frac{1}{2} (\frac{\delta}{2C_0K})^{\frac{1}{\alpha}}$ and $\tilde{r}_k=\tilde{r}_0 e^{-30k\varepsilon/\alpha}$. Since  $y\in B( f^{m+1}x, \tilde{r}_k e^{-30(m+1)\vep/\alpha})$, it follows from Lemma \ref{lem:non-degenerate-neighborhood} that $$ K d(\tilde{y}, f^mx)^{\alpha} <  \frac{\delta}{  2C_0e^{3k\varepsilon}e^{4m\varepsilon}},$$  which implies  \begin{eqnarray*}
	\| D_y\tilde f_{(x,m)} - A_{x,m}\|''<\delta.
\end{eqnarray*}
\end{proof}

By taking local charts, one may identify $E_m(y)=E_m(x)$ and $H_m(y)=H_m(x)$ for all $y$ in a small neighborhood of $x$. By Lemma \ref{prop:f-under-local-chart}, for  any small $\delta>0$, the following 
estimates hold for any $y\in   B( f^{m+1}x, \tilde{r}_k e^{-30(m+1)\vep/\alpha})$:
\begin{eqnarray*}
&&\| D_y\tilde f_{(x,m)}v\|''\geq e^{-a-3\vep}\|v\|'',\quad \forall v\in E_m(y),\\
&&\| D_y\tilde f_{(x,m)} v\|'' \leq e^{-b+3\vep}\|v\|'', \quad \forall v\in H_m(y).
\end{eqnarray*}

Given a splitting $F=F_1\oplus F_2$ of a Euclidean space $F$ with norm $\|\cdot\|$, and a parameter $\xi>0$,  denote by $Q_{\|\cdot\|}(F_1,\xi)$ the cone of width $\xi$ of $F_1$ with respect to $\|\cdot\|$, that is,
\[Q_{\|\cdot\|}(F_1,\xi)=\{v=v_1+v_2\in F:\,v_{1}\in F_1,\,v_2\in F_2,\ \|v_{2}\|\le \xi \|v_{1}\|\}.\] 

\begin{Lem}\label{cone} 
For any $\xi>0$, 	there exists  $\delta>0$ such that for any $x\in\widetilde\Lambda_{a,b,k}$,   $m\geq0,$  and $y\in   B( f^{m+1}x, \tilde{r}_k e^{-30(m+1)\vep/\alpha})$, \begin{eqnarray*} 
D_y\tilde{f}_{(x,m)}(Q_{\|\cdot\|_{y}''}(E_y, \xi))\subset   Q_{\|\cdot\|_{\tilde{f}_{(x,m)}(y)}''}(E_{\tilde{f}_{(x,m)}(y)}, \xi).  \end{eqnarray*}
	\end{Lem}

\begin{proof} 
Let $v=v_1+v_2 \in Q_{\|\cdot\|_{y}''}(E_y, \xi)$ with $v_1\in E_m(y)$, $v_2\in H_m(y)$. Then
\begin{eqnarray*}\| A_{x,m}v_1\|''\geq e^{-a-2\vep}\|v_1\|'',\quad
	\| A_{x,m} v_2\|'' \leq e^{-b+2\vep}\|v_2\|''.
    \end{eqnarray*}
Denote $w= (D_y\tilde f_{(x,m)} - A_{x,m})v$. By Lemma \ref{prop:f-under-local-chart}, it holds that 
\begin{eqnarray*}
\| w\|''\le \delta \|v\|'' \le \delta(1+\xi)\|v_1\|''.  
\end{eqnarray*}	
By the invariance of  $E_m(x)$ and $H_m(x)$ under $A_{x,m}$,  write
\begin{eqnarray*}    D_y\tilde f_{(x,m)}v =  A_{x,m}v +w=  (A_{x,m}v_1+w_1)+( A_{x,m}v_2+w_2).   \end{eqnarray*}
where $w_1\in E_m(x)$, $w_2\in H_m(x).$
It holds that  
\begin{eqnarray*}     &&\|A_{x,m}v_2+w_2\|''\le     e^{-b+2\vep} \|v_2\|''+  \delta (1+\xi)  \|v_1\|'',\\
&&\|A_{x,m}v_1+w_1\|'' \ge     e^{-a-2\vep} \|v_1\|''-  \delta (1+\xi)  \|v_1\|''.
\end{eqnarray*}
Therefore, 
	\begin{eqnarray*}    \frac{\|A_{x,m}v_2+w_2\|''}{\|A_{x,m}v_1+w_1\|''}  \le     \frac{ e^{-b+2\vep} }{ e^{-a-2\vep}- \delta (1+\xi)} \xi+   \frac{ \delta (1+\xi) }{ e^{-a-2\vep}- \delta (1+\xi)}.
	\end{eqnarray*}
	By choosing $\delta>0$ sufficiently small, the right-hand  side above is less than $\xi$. This completes the proof. 
\end{proof}

\begin{Thm}[Stable manifold theorem for $C^{1+\alpha}$ maps]\label{lem:local_manifolds}
There exist $\lambda', \tilde{\alpha}\in(0,1)$ and  $\delta>0$ such that for any $x\in\tilde\Lambda_{a,b,k},$ there exists a   sequence of embedded submanifolds $W^s_{\loc}(f^mx)\subseteq B(f^mx, \tilde{r}_k e^{-30m\vep/\alpha})$, $m=0, 1,\cdots$,   such that
\begin{itemize}\item[(i)] 
$W^s_{\loc}(f^mx)$ contains a ball with radius $\frac{1}{2}\tilde{r}_k e^{-30m\vep/\alpha} $ with respect to $d''_{W^s}$ in $W^s_{\loc}(f^mx)$; 
\item[(ii)] 
 $\|Df\mid_{W^s_{\loc}(f^mx)} \|''\le  \lambda'$;
\item[(iii)]  
$f(W^s_{\loc}(f^mx))\subset W^s_{\loc}(f^{m+1}x)$;
\item[(iv)]$W^s_{\loc}(x)$ is continuous with respect to  $x\in \tilde \Lambda_{a,b,k}$;
\item[(v)] 
$Df\mid_{W^s_{\loc}(f^mx)}$, $\det(Df\mid_{W^s_{\loc} (f^mx)})$ and $TW^s_{\loc}(f^mx)$ are uniformly  $\tilde{\alpha}$-H\"older continuous  with respect to $\|\cdot\|''$, i.e.,  the coefficient $\tilde{C}_0'$ is  independent of $x$ and $m$. 
\end{itemize}
\end{Thm}

\begin{proof} 
Given $x\in\tilde\Lambda_{a,b,k},$ for any $y\in   B( f^{m+1}x, \tilde{r}_k e^{-30(m+1)\vep/\alpha})$,
	\begin{eqnarray*}\| D_y\tilde{f}_{x,m} \|''&=&\sup_{v\neq 0} \frac{\|D_y\tilde{f}_{x,m}(v)\|''}{\|v\|''} =\sup_{w\neq 0} \frac{\|w\|''}{\|D_{f^mx}w\|''}\\[2mm]
		&\le &\sup_{w\neq 0} \frac{ C_0'e^{(k+2m)\vep}  \|w\|}{\frac{1}{2}\|D_{f^mx}w\|}\\[2mm]
		&=&    2 C_0'e^{(k+2m)\vep} \frac{1}{\sigma(D_{f^mx} f)}\\[2mm]
		&\le&    2 C_0'e^{(k+2m)\vep} e^{(k+m)\vep}=  2 C_0'e^{(2k+3m)\vep}
	\end{eqnarray*}
    where the inequalities follow from Proposition \ref{prop:new_norm} and Lemma \ref{Pesin-set}.

    By  \eqref{holder_2norm}, for any $x\in\tilde\Lambda_{a,b,k},$ $Df$ is uniformly  H\"older continuous with respect to $\|\cdot\|''$ in the $e^{-\frac{3(k+m)\vep}{\alpha}}$-neighborhood of $f^mx$, with   H\"older exponent $\alpha/2$ and coefficient $\tilde{C}_0$. Thus, for any $y,z \in B( f^{m+1}x, \tilde{r}_k e^{-30(m+1)\vep/\alpha})$,   
	\begin{eqnarray*}\|D_y\tilde{f}_{x,m}-D_z  \tilde{f}_{x,m}\|''
    &=&\|D_y\tilde{f}_{x,m}\circ(D_{\tilde f_{x,m}(y)}f-D_{\tilde f_{x,m}(z)}{f})\circ D_z  \tilde{f}_{x,m}\|''\\
		&\le & \| D_y\tilde{f}_{x,m} \|'' \| D_{\tilde{f}_{x,m}(y)} f-D_{\tilde{f}_{x,m}(z)} f \|'' \|  D_z  \tilde{f}_{x,m} \|''	\\[2mm] 
		&\le &  (2 C_0'e^{(2k+3m)\vep})^2  \tilde{C}_0  d''(\tilde{f}_{x,m}(y),\tilde{f}_{x,m}(z))^{\frac{\alpha}{2}}	\\[2mm]
		&\le & 	 (2 C_0'e^{(2k+3m)\vep})^2 \tilde{C}_0     (2 C_0'e^{(2k+3m)\vep})^{\frac{\alpha}{2}} d''(y,z)^{\frac{\alpha}{2}}\\[2mm]
		&\le& \hat{C}_0 d''(y,z)^{\frac{\alpha}{4}}\end{eqnarray*}
for some constant $\hat{C}_0$. Therefore, $D\tilde{f}_{x,m}$ is  H\"older continuous in the $\tilde{r}_k e^{-30(m+1)\vep/\alpha}$-neighborhood of $f^{m+1}x$ with respect to the norm $\|\cdot\|''$ , with H\"older exponent $\alpha/4$ and coefficient $\hat{C}_0$. Thus,  $\tilde{f}_{x,m}$ is a $C^{1+\alpha/4}$ local diffeomorphism onto its image, with a uniform H\"older coefficient.  
    
By applying the exponential map, one may assume that the setting is in $\R^{\mathsf d}.$
  Let $\xi$ and $\delta$ be as in Lemma \ref{cone}. For $x\in\widetilde\Lambda_{a,b,k}$ and   $m\geq0$,  extend $\tilde{f}_{x,m}\mid_{B( f^{m+1}x, \tilde{r}_k e^{-30(m+1)\vep/\alpha})}$ to a uniformly  $C^{1+\frac{\alpha}{4}}$ diffeomorphism $\tilde{F}_{x,m}: \mathbb{R}^{\mathsf d} \to \mathbb{R}^{\mathsf d} $ such that
	\begin{itemize}
		\item[-] $\tilde{F}_{x,m}(y)=\tilde{f}_{x,m}(y)$\quad for all $y\in B( f^{m}x, \tilde{r}_k e^{-30m\vep/\alpha})$;\\
		\item[-] $\|D_y\tilde{F}_{x,m}-A_{x,m}\|''\leq 2 \delta$\quad for all $y\in \mathbb{R}^{\mathsf d}$.                  
	\end{itemize}
	By choosing $\delta>0$ sufficiently small,  the conclusions of  Lemma  \ref{cone}  apply to $\tilde{F}_{x,m}$ on $\R^{\mathsf d}$.  

	Consider the disjoint union $\coprod_{i=0}^{+\infty} \{f^i(x)\}\times \mathbb{R}^{\mathsf d}$ with the discrete topology. 
	Note that the global splitting $\coprod_{i=0}^{+\infty} \{f^i(x)\}\times \mathbb{R}^{\mathsf d}=\coprod_{i=0}^{+\infty} \{f^i(x)\}\times (E_{f^i(x)}\oplus H_{f^i(x)})$ is dominated with respect to the extended map $\tilde{F}_{x,m}$.  By \textsection 5  of  \cite{HPS}   or (vi) of Proposition 3.1 of \cite{BW}, one can obtain
	a family $ \{ \mathcal{Y}^{s}_{x,m}:\,m\ge 1\}$ of global invariant   $C^{1+\tilde{\alpha}}$ (for some $\tilde{\alpha}>0$) submanifolds in $\mathbb{R}^{\mathsf d}$ with uniform coefficients, given as $C^{1+\tilde{\alpha}}$ graphs over $H_{f^i(x)}$,  such that for all $m\ge 0$,
	\begin{eqnarray*} &&\{f^m(x)\}\times \mathcal  Y^{s}_{x,m}=\bigcap_{n=m}^{+\infty}  (\tilde{f}_{x,m} \circ \cdots \circ \tilde{f}_{x,n})\left(\{f^{n+1}(x)\}\times Q_{\|\cdot\|_{x,n+1}''}(E_{f^{n+1}(x)}, \xi)\right),\\
    \text{and}\\
		&&\ T_y\mathcal Y^{s}_{x,m} \subset Q_{\|\cdot\|_{x,m}''}(E_{f^m(x)},\xi), \quad \forall y\in \mathcal  Y^{s}_{x,m}.
        \end{eqnarray*}
	In particular, the identity $\tilde{f}_{x,m+1} (\{f^{m+1}(x)\}\times \mathcal Y^{s}_{x, m+1}) =\{f^{m}(x)\}\times \mathcal Y^{s}_{x, m}$ holds.    
    
    Since $\tilde{F}_{x,m}$ coincides with $\tilde{f}_{x,m}$ on $B( f^{m}x, \tilde{r}_k e^{-30m\vep/\alpha})$, which certainly contains the ball at $f^mx$ with radius $\frac 12 \tilde{r}_k e^{-30m\vep/\alpha}$ under $d''_{W^s}.$
    Recalling that $E_m(x)$ depends continuously with respect to $x\in \tilde \Lambda_{a,b,k}$, the proof is concluded with the local stable manifold given by 
    \begin{eqnarray*}
        W^s_{\loc}(f^mx)=\mathcal Y^s_{x,m}\cap B( f^{m}x,   \tilde{r}_k e^{-30m\vep/\alpha}),
    \end{eqnarray*}
    and $\lambda'=e^{a+3\vep}$.  
    Note that  (v) is obtained by the uniform H\"older continuity of $\mathcal Y^s_{x,m}$, with the uniform constants denoted by $\tilde C_0'$.   	
	\end{proof}

Let $\Delta_2=\Delta_1\cap\Delta^\prime$, where (recall) $\Delta^\prime\subseteq M$ is the $\mu$-full measure set on which \eqref{Delta'} holds. 
For $x\in\Delta_2,$ the  stable manifold  of $f$ at $x$ is defined by
\[W^s(x)=\bigcup\nolimits_{n\geq0}(f^{-n}W^s_{\loc}(f^nx)\setminus \Sigma_{f^n}),\]
which is an  immersed submanifold of $M.$  For $x\in I,$ define $W^s(x)=\{x\}.$ 
Since $f$ is non-invertible, $W^s(x)$ is not necessarily connected. Recall that, as introduced in condition (H)', $V^s(x)$ denotes the arc connected component of $W^s(x)$ containing $x$. 

A key mechanism in proving the entropy formula \eqref{pesin eq} is the construction of an appropriate partition associated with the unstable manifolds, through which local expansions are linked to  positive Lyapunov exponents \cite{LY1}. Since 
the entropy formula to be established in \eqref{eq non-inv of folding type} concerns the backward process through negative Lyapunov exponents, a measurable partition associated with the stable manifolds is required instead. 

Let $\mathcal B_\mu(M)$ denote the 
completion of the Borel $\sigma$-algebra with respect of $\mu$. Then $(M,\mathcal B_\mu(M),\mu)$ is a Lebesgue space; that is, it is isomorphic to the unit interval $[0,1]$ with the Lebesgue measure, possibly union a countable collection  of atoms.   

\begin{Def}\label{Def}
{\rm A measurable partition $\xi$ of $(M,\mathcal B_\mu(M),\mu)$ is said to be subordinate to the $W^s$-manifolds of $(f,\mu)$ if for $\mu$-a.e. $x,$ the element of $\xi$ containing $x,$ denoted as $\xi(x),$ satisfies
$\xi(x)\subset W^s(x)$ and  contains an open neighborhood of $x$ in $V^s(x)$ in the submanifold topology of $V^s(x)$. A measurable partition $\xi$ is said to be increasing if $f^{-1}\xi\preceq\xi$, 
where for any two measurable partitions $\zeta_1$ and $\zeta_2$, the notation $\zeta_1\preceq\zeta_2$ means that each element of $\zeta_2$ is contained in some element of $\zeta_1$.
}
\end{Def}
\vspace{1mm}
 
For the Lebesgue space $(M,\mathcal B_\mu(M),\mu)$, any measurable partition $\xi$ naturally induces a system of conditional measures  
$\{\mu^{\xi}_x\}_{x\in M}$, each supported on the elements of the partition. More precisely, 
denote by  $M/ \xi$ the   factor-space of $M$ modulo $\xi$, i.e., each point  of $M/\xi$ is an element of $\xi$.  Let  $p: x\mapsto \xi(x),$ that is, map $p$ takes each point $x\in M$ to the element $\xi(x).$ Note that $\mu$ naturally induces a measure $\mu_{\xi}$ on $M/ \xi$ by  $\mu_{\xi}(Z)=\mu(p^{-1}(Z))$. Moreover, there exists a system of conditional measures    $\{\mu^{\xi}_{B}: B\in \xi\}$ satisfying
\begin{eqnarray}\label{conditional_measure}
  \int_M \phi(x) d\mu= \int_{M/\xi} \int_{B} \phi(x) d \mu^{\xi}_{B}(x)d\mu_{\xi}(B),\quad \forall\phi\in L^1(\mu),
\end{eqnarray}
 where $L^1(\mu)$ denotes the space of integrable functions with respect to $\mu.$ 
We denote $\mu^{\xi}_{x}=\mu^{\xi}_{B}$ whenever $x\in B$. For further details on conditional measures, the reader is referred to \textsection 1 of \cite{Rokhlin}.

For each $x\in\Delta_0\backslash I$, recall that $\lambda^s_x$ denotes the Lebesgue measure on $W^s(x)$ induced by the inherited Riemannian structure. For $x\in I,$ since  $W^s(x)=\{x\},$ we simply let $\lambda^s_x=\delta_x$.
If $\mu$ is an inverse SRB measure, then  for every measurable partition $\xi$ subordinate to the $W^s$-manifolds, it holds that $\mu^{\xi}_x\ll\lambda^s_x$ for $\mu$-a.e. $x\in M,$ where $\mu^\xi_x$ denotes  the conditional measure of $\mu$ on $\xi(x)$.

With the degeneracy of $f$  appropriately controlled as in Lemma \ref{lem:non-degenerate-neighborhood}, an increasing measurable partition
subordinate to the stable manifolds can be constructed. The construction  is an adaptation of \cite {Le-S} in the setting of degeneracy (see also \cite{LQ_book}).  
\begin{Lem}\label{lem:measurable-partition}
There exists an increasing measurable partition $\eta$ that  is subordinate to the $W^s$-manifolds. 
\end{Lem}
\begin{proof} 
For simplicity, denote $\tilde\Lambda_{k}=\tilde \Lambda_{a,b,k}$, where $\tilde \Lambda_{a,b,k}$ is given in Lemma \ref{Pesin-set}. 
By Theorem  \ref{lem:local_manifolds}(ii), for any $x\in \tilde\Lambda_{k} $ and $y, z\in W^s_{\loc}(x)$, it holds for every $n\ge0$ that 
\begin{eqnarray*}
d_{W^s}(f^{n}y,f^{n}z)\le 2d''_{W^s}(f^{n}y,f^{n}z)
\le 2(\lambda')^n  d''_{W^s}(y,z)
\le2 C_0' e^{k\vep}(\lambda')^nd_{W^s}(y,z).    
\end{eqnarray*}

For the  increasing sequence of subsets $\{\tilde\Lambda_{k}\}_{k\ge1}$ and any prescribed decreasing sequence of numbers 
$\{\ell_k\}_{k\ge1}$ (to be specified later in Lemma 3.2), there exist sequences of positive constants $\{\vep_k\}_{k\ge1}$ and $\{R_k\}_{k\ge1}$ such that
    \begin{itemize}
	\item[(i)] For each $x\in \tilde\Lambda_{k}$ and any $y\in\tilde \Lambda_{k}\cap B(x,\varepsilon_k r)$, the following holds: for any    
    $0<r\le R_k$, the intersection $W^s_{\loc}(y)\cap B(x,r)$ is connected, and its diameter along the stable manifold is bounded by $\ell_k,$ that is, 
    \[\diam_{W^s}(W^s_{\loc}(y)\cap B(x,r)) \le\ell_k ;\]

\item[(ii)] The map $y\mapsto W^s_{\loc}(y)\cap B(x,R_k)$ is continuous for $y\in B(x,\varepsilon_k R_k)\cap \tilde\Lambda_{k}$, where the space of subsets of $B(x,R_k)$ is endowed with the Hausdorff topology;
\vspace{1mm}

\item[(iii)] For any  $0<r\le R_k$,  any $x\in\tilde\Lambda_k$ with $k\ge1,$ define 	$$S(x,r)=\bigcup\nolimits_{y\in \tilde\Lambda_k\cap B(x,\varepsilon_k r)} W^s_{\loc}(y)\cap B(x,r).$$
Then for any two points $z_1,z_2\in S(x,r)$ which do not lie in the same local stable  leaf within $S(x,r),$
it holds that $d_{W^s}( z_{1}, z_{2})\ge 2r$. 
\end{itemize}

Without loss of generality, assume that  $\mu$ is ergodic. Choose $\hat{x}\in  \supp (\mu|_{\tilde \Lambda_{k_0}})$  for  some $\tilde \Lambda_{k_0}$ with  $\mu(\tilde \Lambda_{k_0})>0$. Then  $\mu(S(\hat{x},r))>0$ for all $r>0$.   For any $0<r\le R_{k_0}$,  define the measurable partition 
$$\xi_r(y)=
\begin{cases}W^s_{\loc}(y)\cap B(\hat{x},r),&\text{if}\ y\in S(\hat{x},r), \\[2mm]
M\setminus S(\hat{x},r), &\text{if}\ y\not\in S(\hat{x},r).\end{cases}$$
Define 
\begin{eqnarray*}
\xi^s_r\,\, :=\,\,\bigvee\limits_{n=0}^{+\infty}(f^{-n}\xi_r\setminus\Sigma_{f^n}),\quad  S^s_r:=\bigcup\limits_{n=0}^{+\infty}(f^{-n}S(\hat{x},r)\setminus \Sigma_{f^n}),
\end{eqnarray*}
where for a partition $\xi$, write $\xi\setminus\Sigma_{f^n}=\{A\setminus \Sigma_{f^n}:\,\,A\in\xi\}$. 
Since $\mu$ is ergodic and $\mu(S^s_r)>0,$ it follows that $\mu(S^s_r)=1$.

In the following, we show that for $\mu$-a.e.\ $z$, the partition element $\xi^s_r(z)$ contains a neighborhood of $z$ along the stable manifold. It is clear that for any $z\in S^s_r$, there exists some $n\in\N$ such that
$$\xi^s_r(z)\subset f^{-n}W^s_{\loc}(f^{n}z)\subset W^s(z).$$
 For $z\in\bigcup_k \tilde \Lambda_k$,  define
$$\kappa(z)=\inf\{k:\,z\in\tilde \Lambda_{k}\}$$
and  denote 
$$\beta_r(z)=\inf_{n\ge0}\Big{\{}\frac{1}{2}\tilde{r}_{\kappa(z)},\,\,\frac{1}{4 C_0' e^{\kappa(z)\vep}}d(f^{n}z,\,\,\partial B(\hat{x},r)) (\lambda')^{-n},\frac{r}{2 C_0' e^{\kappa(z)\vep}}\Big{\}}.$$

We now show that $\xi^s_r(z)$ contains a neighborhood of $z$ with size  at least $\beta_r(z)$.
First, we show that $\beta_r(z)>0$ for $\mu$-a.e. $z$. Define a finite non-negative measure $\nu$ on $[0,R_k]$ as follows:
$$\nu(A)=\mu(\{z\in M:\,d(\hat{x},z)\in A\}).$$ 
By Proposition 3.2 of \cite{Le-S},  there exists $0<r<R_k$ such that $$\sum_{n=0}^{\infty}\nu([r- (\lambda')^n,r+ (\lambda')^n])<+\infty,$$
where such values of $r$ form a subset of $(0,R_k)$ of full Lebesgue measure. It follows from the invariance of $\mu$ that 
$$\sum_{n=0}^{\infty}\mu(\{z: d(f^nz,\partial B(\hat{x},r))\in [r- (\lambda')^{n},r+ (\lambda')^n]\})<+\infty.$$
By the Borel-Cantelli lemma,  for $\mu$-a.e. $z$, the  number of $n$ such that  $$d(f^nz,\partial B(\hat{x},r))\in [r- (\lambda')^{n}, r+ (\lambda')^{n}]$$ 
is finite. This implies that  $\beta_r(z)>0$ for $\mu$-a.e. $z.$

Next, for any $v\in W^s_{\loc}(z,  \beta_r(z))$, we claim that \begin{eqnarray}\label{eql}\xi^s_r(f^nz)= \xi^s_r(f^nv),\quad\forall n \ge0.
\end{eqnarray}   
Indeed,  by the definition of  $\beta_r(z)$,  we have for all $n\ge 0$, 
\begin{eqnarray}\label{1small}
&&d_{W^s}(f^nv, f^nz)\le \beta_r(z)  2 C_0' e^{\kappa(z)\vep}(\lambda')^n \le  \frac{1}{2} d(f^{n}z,\,\,\partial B(\hat{x},r)),\\[2mm]
\label{2small}
&&d_{W^s}(f^nv, f^nz) \le  \beta_r(z) 2 C_0' e^{\kappa(z)\vep}(\lambda')^n \le r.
\end{eqnarray}
By the choice of $S(\hat{x},r)$,   
both $f^nv$ and$f^nz$ either belong to $S(\hat{x},r)$ or lie outside it. Then
 \begin{itemize}
\item[-] If  $ f^nz, f^nv\in S(\hat{x},r)$, then   (\ref{eql}) follows from  (\ref{1small})  and  (\ref{2small}).
\item[-]  If  $ f^nz, f^nv\notin S(\hat{x},r)$, then   (\ref{eql}) follows from the definition of  $\xi^s_r$.
\end{itemize}
This completes the proof of Lemma \ref{lem:measurable-partition} by letting $\eta=\xi^s_r$. 
\end{proof}

\section{Proof of  Theorem~\ref{thm1} -- necessity }\label{nes}
This section proves  the ``necessity  part" of Theorem \ref{thm1}, i.e., if $\mu$  is an inverse SRB measure, then the entropy formula of folding type \eqref{eq non-inv of folding type} and (H)" hold. 

We begin with the following  fact about conditional entropy. 
\begin{Lem}\label{lem:conditional-entropy-upper-bound}
Let $(X, \mathcal B, \mu)$ be a Lebesgue space, $f: X\to X$ be a $\mu$-preserving transformation, and $\eta$ be an 
 increasing measurable partition of $X$.
Then
\begin{eqnarray}\label{conditional-eta-upper-bound}
H_\mu(\eta|f^{-1}\eta)\leq h_\mu(f).
\end{eqnarray}
In particular, 
\begin{eqnarray}\label{folding-entropy-upper-bound}
F_\mu(f)=H_\mu(\epsilon|f^{-1}\epsilon)\leq h_\mu(f)
\end{eqnarray}
where $\epsilon$ denotes the measurable partition into single points. 
\end{Lem}

\begin{proof} 
Take a sequence of finite partitions $\{\xi_m\}_{m\geq1}$ of $X$ such that $\xi_m$ is increasing to $\eta$ as $m\to\infty.$
For any $m\ge1,$
\begin{eqnarray*}
h_\mu(f,\xi_m)=\lim\limits_{n\to\infty}H_\mu(\xi_m|\bigvee_{i=1}^nf^{-i}\xi_m)
\geq\lim\limits_{n\to\infty}H_\mu(\xi_m|\bigvee_{i=1}^nf^{-i}\eta),
\end{eqnarray*}
where the last inequality comes from that $\xi_m\preceq\eta$ for each $m\ge1.$ Note that $f^{-1}\eta\preceq\eta$ implies $\bigvee_{i=1}^nf^{-i}\eta\preceq f^{-1}\eta$ for any 
$n\ge1,$ and 
since $f^{-1}\eta\preceq\bigvee_{i=1}^nf^{-i}\eta,$ we have $\bigvee_{i=1}^nf^{-i}\eta=f^{-1}\eta$. Thus,
$$H_\mu(\xi_m|\bigvee_{i=1}^nf^{-i}\eta)=H_\mu(\xi_m|f^{-1}\eta),\quad\forall n\geq1,$$
and hence
\begin{eqnarray}\label{ineq:conditional-eta}
h_\mu(f,\xi_m)\geq\lim\limits_{n\to\infty}H_\mu(\xi_m|\bigvee_{i=1}^nf^{-i}\eta)=H_\mu(\xi_m|f^{-1}\eta).
\end{eqnarray}
Let $m\to\infty$ on both sides of \eqref{ineq:conditional-eta},  and together with the fact that $h_\mu(f,\xi_m)\leq h_\mu(f),\forall m\geq1,$ 
\eqref{conditional-eta-upper-bound} is obtained. 
\end{proof}
\begin{Rmk}
{\rm For the special case of $C^{1+\alpha}$ interval maps, the equality $F_\mu(f)=h_\mu(f)$ is established by the authors (see Theorem 4.1 in \cite{LW2}) with the ``$\le$" direction obtained using a similar approach. 
}
\end{Rmk}

Let $\eta$ be the measurable partition given in Lemma \ref{lem:measurable-partition}. Denote  
\[D^s(z)=|\det(D_zf|_{E^s(z)})|,\quad\text{ $\forall z\in\eta(x)$}.\] 
The following convergence result is crucial for controlling the dynamics along stable manifolds.  

\begin{Lem}\label{prop:bound}
	For {$\mu$-a.e. $x\in\Delta_2$}, the infinite product
	\begin{eqnarray}\label{convergence}
	\prod_{n=0}^\infty\dfrac{D^s(f^nz)}{D^s(f^nx)}<\infty,
	\end{eqnarray}
converges and is uniformly bounded  for all $z\in\eta(x).$
\end{Lem}
\begin{proof} 
Without loss of generality, assume that $\mu$ is ergodic with $q$ Lyapunov exponents $-\infty\le\lambda_1\le\lambda_2\le\cdots\le\lambda_q<\infty$, and there exists $1\le s\le q$ such that $\lambda_s<0\le\lambda_{s+1},$ where we let $\lambda_{s+1}=\lambda_{q}+1$ if $s=q.$ Hence, with $a=\lambda_s/2$ and $b=\lambda_{s+1}/2,$ we have $\mu(\tilde\Lambda_{a,b})=1.$

Note that in the proof of Lemma \ref{lem:measurable-partition},  any element of $\eta$ belonging to the $\mu$-positive measurable set $	S(\hat{x},r)=\bigcup_{y\in \tilde\Lambda_{k_0}\cap B(\hat{x},\varepsilon_k r)} W^s_{\loc}(y)\cap B(\hat{x},r)$ is properly controlled. By the ergodicity of $\mu$ and the properties of the set $S(\hat{x},r)$, for $\mu$-a.e. $x$, 
there exists a finite $n_0\in\N$ and certain $y \in\tilde\Lambda_{k_0}$ such that 
\[f^{n_0}(\eta(x))\subseteq W^s_{\loc}(y)\cap B(\hat{x},r).\]
By the construction of Pesin set in 
Lemma \ref{Pesin-set}(iv), we have
\begin{eqnarray}\label{Ds}D^s(y)\ge  e^{-sk_0\varepsilon}. \end{eqnarray}    
	
For any $z\in W^s_{\loc}(y)\cap B(\hat{x},r),$ 	
let $A=T_y W^s_{\loc}(y)$ and $B=T_zW^s_{\loc}(y)$. The distance between $A$ and $B$ is defined as 
	$$d(A, B):=\max\Big{\{}  \max_{v\in A, \|v\|=1} d(v, B),\,   \max_{v\in B, \|v\|=1} d(v, A)\Big{\}},$$
	where $$d(v, A)=\min_{w\in A} \|v-w\|.$$
Similarly,  denote 	$d''(A, B)$ as the distance with respect to $\|\cdot\|''$. By definition, it satisfies 
$$d(A, B)\le 2C_0' e^{k_0\vep} d''(A, B).$$

Since $f$ is $C^{1+\alpha}$, there exists $C_1>0$ such that  
\begin{eqnarray*}
 |D^s(y)-D^s(z)|&=&|\det(Df\mid_A)-\det(Df\mid_B)| \nonumber\\
&\le& C_1d^{\alpha}(A, B)\le C_1 (2C_0' e^{k_0\vep} d''(A, B))^{\alpha}.   
\end{eqnarray*}
By Theorem \ref{lem:local_manifolds} (v), we have 
\begin{eqnarray*}d''(A, B)\le \tilde{C}_0^\prime (d''(y,z))^{\tilde{\alpha}}\le \tilde{C}_0^\prime( C_0' e^{k_0\vep} d(y,z))^{\tilde{\alpha}}.
\end{eqnarray*}
Therefore,  
\begin{align}\label{Ds-d}
|D^s(y)-D^s(z)|\le C_1 (2\tilde{C}_0^\prime)^{\alpha}(C_0' e^{k_0\vep} )^{\alpha(1+\tilde{\alpha})}d^{\alpha \tilde{\alpha}}(y,z).   
\end{align}
It follows that 
	\begin{eqnarray*}
		\Big|\dfrac{D^s(z)}{D^s(y)}-1\Big|&=&\Big|\dfrac{D^s(z)-D^s(y)}{D^s(y)}\Big|\\
        &\le&\dfrac{C_1 (2\tilde{C}_0^\prime)^{\alpha}(C_0' e^{k_0\vep} )^{\alpha(1+\tilde{\alpha})}d^{\alpha \tilde{\alpha}}(y, z)}{e^{-sk_0\varepsilon}}\le \frac12,
	\end{eqnarray*}
	provided that the constants  $\{\ell_k\}_{k\ge1}$  in (i) of  the construction of $\eta$ in Lemma \ref{lem:measurable-partition} satisfy
    \begin{eqnarray*}
     \ell_k\le \Big(\frac{1}{2}\frac{e^{-sk\varepsilon}}{C_1 (2\tilde{C}_0^\prime)^{\alpha}(C_0' e^{k\vep} )^{\alpha(1+\tilde{\alpha})} }\Big)^{\frac{1}{\alpha\tilde\alpha}},\quad\forall k\ge1.  
    \end{eqnarray*}
     In particular, this implies  
	\begin{eqnarray}\label{Dz}
	\Big|\dfrac{D^s(z)}{D^s(y)}\Big|\ge 1/2,\quad\forall z\in  W^s_{\loc}(y)\cap B(\hat{x},r).  
	\end{eqnarray}

    Now, for any $z\in\eta(x),$ since both $f^{n_0}x$ and $f^{n_0}z$ belong to $W^s_{\loc}(y)\cap B(\hat{x},r),$ the differential mean value theorem implies that 
	there exists $v_z\in W^s_{\loc}(y)\cap B(\hat{x},r)$ such that  
\begin{align*}
&|\log D^s(f^{n_0}x)-\log D^s(f^{n_0}z)| \\&\le\dfrac{1}{D^s(v_z)}|D^s(f^{n_0}x)-D^s(f^{n_0}z)|\\
&\le  \dfrac{1}{D^s(v_z)} C_1 (2\tilde{C}_0^\prime)^{\alpha}(C_0' e^{k_0\vep} )^{\alpha(1+\tilde{\alpha})}d^{\alpha\tilde{\alpha}}_{W^s}(f^{n_0}x,f^{n_0}z) \quad (\text{by}\,\, (\ref{Ds-d})).
\end{align*}  
 Thus,   \begin{align*}
 |\log D^s(f^{n_0}x)-\log D^s(f^{n_0}z)| &\le \dfrac{\tilde{C}_1e^{\alpha k_0\varepsilon(1+\tilde\alpha)}d^{\alpha\tilde\alpha}_{W^s}(f^{n_0}x,f^{n_0}z)}{D^s(v_z)}
\end{align*}
  for some constant $\tilde{C}_1>0$,   which, by \eqref{Dz} and \eqref{Ds}, yields 
	\begin{eqnarray*}\label{bound}
		|\log D^s(f^{n_0}x)-\log D^s(f^{n_0}z)|		&\le&\dfrac{2\tilde{C}_1e^{\alpha k_0\varepsilon(1+\tilde\alpha)}d^{\alpha\tilde\alpha}_{W^s}(f^{n_0}x,f^{n_0}z)}{D^s(y)}\\		&\le& e^{sk_0\vep}\cdot 2\tilde{C}_1e^{\alpha k_0\varepsilon(1+\tilde\alpha)}d^{\alpha\tilde\alpha}_{W^s}(f^{n_0}x,f^{n_0}z),
	\end{eqnarray*}
	i.e.,	\begin{eqnarray}\label{ineq:D^s_bound}
	\dfrac{D^s(f^{n_0}z)}{D^s(f^{n_0}x)}\le e^{2\tilde{C}_1 e^{k_0\vep(s+\alpha+\alpha\tilde\alpha)}d_{W^s}^{\alpha\tilde\alpha}(f^{n_0}x,f^{n_0}z)}. 
	\end{eqnarray}
	
We note that for all $n\ge n_0$, $f^ny,f^nx\in f^{n-n_0}(W^s_{\loc}(y))$ always lie in the local stable manifold of $f^{n-n_0}(y)$, which, according to Lemma \ref{Pesin-set}, belongs to the Pesin set $\tilde\Lambda_{a,b,k_0+n-n_0}$. 
Thus, the above analysis  applies  to  all $n\ge n_0$ with $k_0$ replaced by $(k_0+n-n_0).$ Consequently, 
	\begin{eqnarray}\label{D^s(n)}
	\dfrac{D^s(f^{n}z)}{D^s(f^{n}x)}\le e^{2\tilde{C}_1 e^{(k_0+n-n_0)\vep(s+\alpha+\alpha\tilde\alpha)}d_{W^s}^{\alpha\tilde\alpha}(f^{n}x,f^{n}z)},\quad\forall n\ge n_0.
	\end{eqnarray}  
	By the contraction property of local stable manifolds in Theorem \ref{lem:local_manifolds}, there exists $\lambda^\prime\in(0,1)$ such that 
	\[d_{W^s}(f^nx,f^nz)\le 2d''_{W^s}(f^nx,f^nz)\le 2({\lambda'})^{n-n_0}d''_{W^s}(f^{n_0}x,f^{n_0}z),\quad\forall n\ge n_0.\]
	Thus, by letting $\varepsilon>0$ be sufficiently small so that $e^{\vep(s+\alpha+\alpha\tilde\alpha)}{\lambda^\prime}^{\alpha\tilde\alpha}<1$, we obtain
	\begin{eqnarray*}
		\prod_{n=n_0}^\infty\dfrac{D^s(f^nz)}{D^s(f^nx)}&\le& e^{2\tilde{C}_1 e^{k_0\vep(s+\alpha+\alpha\tilde\alpha)}(\sum_{n=n_0}^\infty e^{(n-n_0)\varepsilon(s+\alpha+\alpha\tilde\alpha)}d_{W^s}^{\alpha\tilde\alpha}(f^{n}x,f^{n}z))}\\
		&\le& e^{2\tilde{C}_1 e^{k_0\vep(s+\alpha+\alpha\tilde\alpha)}d_{W^s}''^{\alpha\tilde\alpha}(f^{n_0}x,f^{n_0}z) 2^{\alpha\tilde\alpha}\sum_{n=n_0}^\infty (e^{\varepsilon(s+\alpha+\alpha\tilde\alpha)}{\lambda'}^{\alpha\tilde\alpha})^{(n-n_0)}}<\infty.
	\end{eqnarray*}
	
	Since $f^{n_0}x, f^{n_0} z$ lie in the same disc $W^s_{\loc}(y)\cap B(\hat{x},r)$, whose diameter is uniformly bounded,
	the $d_{W^s}''$-distance between 
	$f^{n_0}x$ and $f^{n_0} z$ is also uniformly bounded. Hence, 
    \begin{eqnarray*}
		\prod_{n=n_0}^\infty\dfrac{D^s(f^nz)}{D^s(f^nx)}<\infty,\quad\text{uniformly bounded  for all $z\in\eta(x).$}
	\end{eqnarray*}
Moreover,  recall  $L=\sup_{x\in M}\|D_xf\|$. Then \begin{eqnarray*}
		\prod_{n=0}^{n_0-1}\dfrac{D^s(f^nz)}{D^s(f^nx)}\le \prod_{n=0}^{n_0-1}\dfrac{L^s}{D^s(f^nx)}<\infty,\quad\text{uniformly  bounded for all}\, z\in\eta(x).
	\end{eqnarray*}    
	Therefore, 
	\begin{eqnarray*}
		\prod_{n=0}^\infty\dfrac{D^s(f^nz)}{D^s(f^nx)}<\infty,\quad\text{uniformly bounded for all $z\in\eta(x).$}
	\end{eqnarray*}
	
This completes the proof of Lemma \ref{prop:bound}.
\end{proof}

\noindent{\it Proof of the necessity in Theorem  \ref{thm1}:}

(i) {\it Proof of the entropy formula of folding type:} 
Since the ``$\le$"  direction   in the formula \eqref{eq non-inv of folding type}  has already been proved by the authors \cite{LW1} for any $C^{1+\alpha}$ map  with invariant measure, it only remains to prove the  ``$\ge$" direction, i.e., 
\begin{eqnarray}\label{eq:only prove}
	h_{\mu}(f)\ge F_{\mu}(f)-\displaystyle\int_{M}\sum\nolimits_{\lambda_i(x)<0} \lambda_i(x)d\mu(x),
\end{eqnarray}
under the condition that $\mu$ has 
absolutely continuous conditional
measure on the stable manifolds.
 
Let $\eta$ be the  measurable partition given by Lemma \ref{lem:measurable-partition}. Since $\mu$ is an inverse SRB measure, there exists, for 
$\mu$-a.e. $x$,  a $\lambda^s_x$-integrable function $ \psi$  on $\eta(x)$ such that 
\begin{eqnarray}\label{h} 
\dfrac{d\mu_x^\eta}{d\lambda^s_x}(y)=\psi(y), \quad  \lambda^s_x{\text{-a.e.}}\  y\in\eta(x).
\end{eqnarray}

For $\mu$-a.e. $x,$ consider the transformation between Lebesgue spaces
\[f_x :=f|_{(f^{-1}\eta)(x)}:\big((f^{-1}\eta)(x),\mu^{f^{-1}\eta}_x\big)\rightarrow\big(\eta(fx),\mu^{\eta}_{fx}\big).\]
By the $f$-invariance of $\mu$, $f_x$ is measure-preserving. 
Moreover, since 
$f_x$ is $\mu_x^{f^{-1}\eta}$-a.e. non-degenerate on $(f^{-1}\eta)(x),$
its Jacobian $J_{f_x}$ is well-defined. According to Lemma 10.5 in \cite{Parry}, the following identity holds
\begin{eqnarray}\label{eq:log(Jac)}
\log J_{f_x}(z)=-\log{\big((\mu_x^{f^{-1}\eta})_{\ _{z}}^{f_x^{-1}\epsilon}}(\epsilon(z))\big),\quad \ \mu_x^{f^{-1}\eta}{\text{-a.e.} }\, z\in (f^{-1}\eta)(x),
\end{eqnarray}
where the notation $(\mu_x^{f^{-1}\eta})_{\ _{z}}^{f_x^{-1}\epsilon}$
refers to the conditional measure by disintegrating
$\mu_x^{f^{-1}\eta}$ along the partition $f_x^{-1}\epsilon$, and $\mu_x^{f^{-1}\eta}$ itself denotes the conditional measure of $\mu$ with respect to the partition $f^{-1}\eta;$ that is, 
it involves a disintegration of a disintegration.

Noting that  $f^{-1}\eta   \preceq f^{-1}\epsilon$,  i.e.,  $f^{-1} \epsilon$ is a finer partition than $f^{-1}\eta$,  it follows from 
\eqref{conditional_measure} that    for any $\phi\in L^1(\mu)$, 
\begin{eqnarray}\label{trans} 
&&\int_M  \phi d\mu=\int_{M/f^{-1}\eta}\int_{(f^{-1}\eta)(x)} \phi d\mu^{f^{-1}\eta}_x d\mu_{f^{-1}\eta}\nonumber\\[2mm]
&&= \int_{M/f^{-1}\eta}\int_{(f^{-1}\eta) (x)/f^{-1}\epsilon} \int_{(f^{-1}\epsilon)(z)}\phi d (\mu^{f^{-1}\eta}_x)_z^{f^{-1}\epsilon}d(\mu^{f^{-1}\eta}_x)_{f^{-1}\epsilon} d\mu_{f^{-1}\eta}.
\end{eqnarray}
On the other hand, by directly applying \eqref{conditional_measure} with $\xi=f^{-1}\epsilon$, one obtains 
\begin{eqnarray}\label{trans1}
\int_M \phi d\mu= \int_{M/f^{-1}\epsilon}\int_{(f^{-1}\epsilon)(z)} \phi d\mu^{f^{-1}\epsilon}_z d\mu_{f^{-1}\epsilon}.
\end{eqnarray}
Combining \eqref{trans} and \eqref{trans1}, and noting the  arbitrariness of $\phi$, one obtains the transitivity property of conditional measures, that is, 
for  $\mu$-a.e. $x$ and $\ \mu_x^{f^{-1}\eta}$-a.e.\, $z\in (f^{-1}\eta)(x)$,  
$$  (\mu^{f^{-1}\eta}_x)_z^{f^{-1}\epsilon}(\epsilon(z)) =\mu^{f^{-1}\epsilon}_z(\epsilon(z)).$$
Since for $\mu$-a.e. $z$ one has
\begin{eqnarray}\label{log_J}
\log J_{f}(z)=-\log\mu_z^{f^{-1}\epsilon}(\epsilon(z)),
\end{eqnarray}
this identity, together with \eqref{eq:log(Jac)}, implies that 
\begin{eqnarray}\label{J=Jx}
\log J_{f_x}(z)=-\log\mu_z^{f^{-1}\epsilon}(\epsilon(z))=\log J_{f}(z),\ \mu_x^{f^{-1}\eta}{\text{-a.e.}}\, z\in (f^{-1}\eta)(x).
\end{eqnarray}
Also note that $f^{-1}\eta  \preceq\eta$. By applying  $\eta$ as $f^{-1}\epsilon$ in (\ref{trans}),  we have for $\mu$-a.e. $x$  and $\mu_x^{f^{-1}\eta}$ -a.e. $ z\in (f^{-1}\eta)(x),$
\begin{eqnarray*}    d\mu^{f^{-1}\eta}_{x}(z)=\mu_x^{f^{-1}\eta}(\eta(z))d\mu^\eta_x(z),\ \ 
\end{eqnarray*}
which, combined with \eqref{h}, yields 
\begin{eqnarray}
J_{f_x}(z)&=& \frac{d\mu_{fx}^{\eta} (fz)}{  d\mu^{f^{-1}\eta}_{x}(z)}\nonumber\\[2mm] 
&=&\dfrac{1}{\mu_x^{f^{-1}\eta}(\eta(z))}\cdot \frac{d\mu_{fx}^{\eta} (fz)}{ d\mu^{\eta}_x(z)}\nonumber\\[2mm] &=&  \dfrac{1}{\mu_x^{f^{-1}\eta}(\eta(z))}\cdot \dfrac{\psi(fz)}{\psi(z)}\cdot\frac{d\lambda^s_{fx}(fz)}{d\lambda^s_x(z)}\nonumber\\[2mm]\label{Jf_x(1)} &=&\dfrac{1}{\mu_x^{f^{-1}\eta}(\eta(z))}\cdot \dfrac{\psi(fz)}{\psi(z)}\cdot\big|\det(D_zf|_{E^s(z)})\big|.
\end{eqnarray}
Taking logarithm on both sides of \eqref{Jf_x(1)} and combining with \eqref{log_J}, we obtain that for $\mu$-a.e. $z,$
\begin{eqnarray}\label{eq:local}
-\log\mu_z^{f^{-1}\epsilon}(\epsilon(z))=-\log\mu_x^{f^{-1}\eta}(\eta(z))+\log\dfrac{\psi(fz)}{\psi(z)}
+\log|\det(D_zf|_{E^s(z)})|.
\end{eqnarray}

To integrate  \eqref{eq:local}, it is necessary to ensure that each term in the equation is $\mu$-integrable. 
To verify this, note that by the definition of conditional entropy,
\[\displaystyle\int-\log\mu_x^{f^{-1}\epsilon}(\epsilon(z))d\mu(x)=H_\mu(\epsilon|f^{-1}\epsilon)\]
and \[\displaystyle\int-\log\mu_x^{f^{-1}\eta}(\eta(z))d\mu(x)=H_\mu(\eta|f^{-1}\eta).\]
Together with Lemma~\ref{lem:conditional-entropy-upper-bound} and the fact that  
$h_\mu(f)<\infty,$ it follows that 
$$-\log\mu_z^{f^{-1}\epsilon}(\epsilon(z))\in L^1(\mu),\ \ -\log\mu_x^{f^{-1}\eta}(\eta(z))\in L^1(\mu).$$
Moreover, by the integrability condition of $\mu$, $\log|\det(D_zf|_{E^s(z)})\big|\in L^1(\mu).$ Hence, 
$
\log\big(\dfrac{\psi\circ f}{\psi}\big)$ is also $\mu$-integrable. 
Then,  Proposition 2.2 of \cite{Le-S} yields that 
\begin{align}\label{integral=0}
\displaystyle\int\log\dfrac{\psi(fz)}{\psi(z)}d\mu(z)=0.  
\end{align}
Also note that 
for $z\in I=\{x\in\Delta_0:\dim E^s(x)=0\}$, \eqref{eq:local}   still holds by simply taking $\eta=\epsilon$ on  $I$ and hence \[\psi(z)=1,\ \log|\det(D_zf|_{E^s(z)})|=0.\]

Now, integrating both sides of \eqref{eq:local}, we obtain 
\begin{eqnarray}\label{eq:integral-form}
H_\mu(\epsilon|f^{-1}\epsilon)=H_\mu(\eta|f^{-1}\eta)+
\displaystyle\int\sum\nolimits_{\lambda_i(z)<0}\lambda_i(z)d\mu(z),
\end{eqnarray}
where the last term on the right-hand side comes from the fact that   $$\displaystyle\int\log|\det(D_zf|_{E^s(z)})|d\mu(z)=\displaystyle\int\sum\nolimits_{\lambda_i(z)<0}\lambda_i(z)d\mu(z).$$
Since, by Lemma~\ref{lem:conditional-entropy-upper-bound}, $H_\mu(\eta|f^{-1}\eta)\le h_\mu(f)$, \eqref{eq:integral-form}  directly yields
\eqref{eq:only prove}. 
\medskip

(ii) {\it Proof of condition (H)'':} Note that for any $z\in\eta(x)$ and $n\ge1,$  $f^nz\in\eta(f^nx)$. Thus, 
\eqref{Jf_x(1)} together with \eqref{J=Jx} yields
\begin{eqnarray*}
	\dfrac{\psi(z)}{\psi(x)}&=&\dfrac{\psi(fz)}{\psi(fx)}\cdot\dfrac{J_{f_x}(x)}{J_{f_x}(z)}\cdot\dfrac{D^s(z)}{D^s(x)}\\
	\vspace{1mm}\\
    &=&\dfrac{\psi(fz)}{\psi(fx)}\cdot\dfrac{J_{f}(x)}{J_{f}(z)}\cdot\dfrac{D^s(z)}{D^s(x)}\\
	\vspace{1mm}\\
	&=&\cdots   \\
	&=& \  \Big(\prod_{k=0}^n\dfrac{J_f(f^kx)}{J_f(f^kz)}  \cdot\dfrac{D^s(f^kz)}{D^s(f^kx)}  \Big)\frac{\psi(f^{n+1}(z))}{\psi(f^{n+1}(x))}    \\
	&=& \lim_{n\rightarrow\infty}\Big(\prod_{k=0}^n\dfrac{J_f(f^kx)}{J_f(f^kz)}\Big)\frac{\psi(f^{n+1}(z))}{\psi(f^{n+1}(x))} \lim_{n\rightarrow\infty}\Big(\prod_{k=0}^n\dfrac{D^s(f^kz)}{D^s(f^kx)}  \Big) \\[2mm] &:=& \Theta(x,z)\lim_{n\rightarrow\infty}\Big(\prod_{k=0}^n\dfrac{D^s(f^kz)}{D^s(f^kx)}  \Big),\quad\text{$\lambda^s_x$-a.e.\ $z\in\eta(x)$},
\end{eqnarray*}
where the convergence is guaranteed  by Lemma \ref{prop:bound}.  

Note that \eqref{sum=1} holds 
as a property of Jacobian.
For $y$ satisfying $\psi(y)=0$, one may choose $J_f$ arbitrarily satisfying (\ref{sum=1}), for completeness.   By the integrability of $\psi(z)$ with respect to $\lambda^s_x$ on $\eta(x)$, condition (H)" is proved. 
\qed

\section{Proof of Theorem \ref{thm1}--  sufficiency}\label{suf}
This section proves  the ``sufficiency part'' of  Theorem \ref{thm1}, that is,  if an $f$-invariant measure $\mu$ satisfies
  the folding-type entropy formula \eqref{eq non-inv of folding type}  and the condition (H)'', then  $\mu$ must be an inverse SRB measure.

Without loss of generality, it may be assumed that $\mu$ is ergodic. 
In fact, given any $f$-invariant measure $\mu$, if the folding-type entropy formula \eqref{eq non-inv of folding type} holds, then it also holds for almost every ergodic components of $\mu$. This follows from the fact that the inequality part of \eqref{eq non-inv of folding type} (i.e.,  the folding-type Ruelle inequality) has been established  
for all $f$-invariant measures \cite{LW1}. Therefore, if the equality \eqref{eq non-inv of folding type} holds for $\mu$, it must also hold for almost all  of its ergodic components. 

When $\mu$ is ergodic,
Theorem \ref{thm:MET} asserts that there exist
$q(q>0)$ Lyapunov exponents $\lambda_i (1\le i\le q)$ such that  
$-\infty\le \lambda_1\le\cdots\le\lambda_q<\infty$. 
Note that if $\lambda_1\ge0,$ then for $\mu$-a.e. $x$, we define 
$W^s(x)=\{x\}$ (which may not be the global manifold), and the conditional measure of $\mu$ on $W^s(x)$ is $\delta_x$.
The conclusion holds automatically. 
In the following, we assume $\lambda_1<0$, and let $1\le s\le q$ be such that $\lambda_s<0\le \lambda_{s+1}$ where we set  $\lambda_{s+1}=\lambda_{q}+1$ if $s=q.$ Note that $\mu(\Lambda_{\lambda_s/2,\lambda_{s+1}/2})=1.$

Recall that   $(\bar M, \tau, \bar{\mu})$ is the lifting system on the inverse limit space $\bar M$ of $(M,f,\mu),$ where $\bar M=\{\bar x=(\cdots, x_{-1},x_0,x_1,\cdots):x_i\in M,fx_i=x_{i+1}, i\in\mathbb Z\}$, and  $\pi:\bar M\to M$, $\bar x\mapsto x_0$ is the natural projection.

The following lemma  gives an additional  property for the 
 measurable partition $\eta$  in Lemma \ref{lem:measurable-partition}.
\begin{Lem}\label{prop:partition_xi}
Let $\eta$ be the  measurable partition in  Lemma \ref{lem:measurable-partition}. Then  $\vee_{n=0}^\infty\tau^n(\pi^{-1}\eta)$ is a partition of $\bar M$ into single points. 
\end{Lem}

\begin{proof}
Note that  $\bar{\mu}(\pi^{-1}(S(\hat{x},r)))=\mu(S(\hat{x},r))>0$.  For $\bar{\mu}$-a.e. $\bar{z}$,  there exists an increasing sequence of integers $\{n_i\}_{i=1}^\infty$ such that $\tau^{-n_i}(\bar{z})\in \pi^{-1}(S(\hat{x},r))$.  If $\bar{v}\in\big(\tau^{n_i}(\pi^{-1}\eta)\big)(\bar{z})$, then  $z_{-n_i}, v_{-n_i}\in W^s_{\loc}(y)$ for some $y\in \tilde\Lambda_{k}$,   
which implies for $-n_i/2\le   t\le 0 $,  
\begin{align*}
d_{W^s}(z_t, v_{t}) &=  d_{W^s}(f^{n_i+t}(z_{-n_i}), f^{n_i+t}(v_{-n_i}))\le  2 C_0' e^{k\vep}(\lambda')^{n_i+t}d_{W^s}(z_{-n_i},v_{-n_i}) \\[2mm] &\le    C_0' e^{k\vep}(\lambda')^{\frac{n_i}{2}}d_{W^s}(z_{-n_i},v_{-n_i})\to 0,\quad \text{as}\,\,n_i\to \infty. \end{align*}
Thus, $\bar{z}=\bar{v} $. 
\end{proof}
\medskip

Now, assuming that condition (H)" holds, let 
\begin{eqnarray*}
   \zeta=\eta\vee \xi,
\end{eqnarray*}
where $\eta$ is as in Lemma \ref{prop:partition_xi}, and $\xi$ is the partition given in (H)''. 
Then $\zeta$ is an 
increasing measurable partition subordinate to the $W^s$-manifolds and satisfies the property in Lemma \ref{prop:partition_xi}.

\begin{Prop}\label{prop:entropy}
For measurable partition $\zeta$, it holds that 
\begin{eqnarray}\label{conditional_entropy}
h_\mu(f)=H_\mu(\zeta|f^{-1}\zeta).
\end{eqnarray}
\end{Prop}
\medskip

The  entropy equation  \eqref{conditional_entropy} 
was established in \cite{LY1} for $C^2$ diffeomorphisms (with $\zeta$ subordinate to the unstable manifolds, though),
and was later generalized to non-degenerate $C^2$ maps in \cite{Liu08} through the inverse limit space. Since $\zeta$ is subordinate to the $W^s$-manifolds, it follows from Lemma \ref{prop:partition_xi} that  $\zeta$ achieves the entropy on stable manifolds. The main task left in proving \eqref{conditional_entropy} is to show that the full entropy is  concentrated on stable manifolds. In Section 2,  Pesin theory for general $C^{1+\alpha}$ maps has been established: 
the Pesin sets and the Lyapunov neighborhoods are constructed on which $f$ exhibits uniform (partial) hyperbolicity for the normalization.  The proof in Section 4.2 of \cite{Liu08} and Proposition 5.1 of \cite{LY1} can therefore be adapted to the present setting by working with the inverse limit space. More specifically, consider a finite entropy partition adapted to the Lyapunov neighborhoods as constructed in Lemma 4.2.1 of \cite{Liu08}. One then considers the following two cases: 
\begin{itemize}
\item[-] In the the hyperbolic case (i.e., when all Lyapunov exponents are non-zero),  Section 4.2 of \cite{Liu08} shows that any point $\bar{y}$ remaining in the same Lyapunov neighborhoods as $\bar x$ under forward iterations must lie on the stable manifold of $\bar x$ within the neighborhood; that is, $y_0\in V^s(x_0)$;
\vspace{1mm}

\item[-] In the non-hyperbolic case (i.e., when some Lyapunov exponents vanish), an argument analogous to Proposition 5.1 of \cite{LY1} applies. One considers a submanifold $\mathcal{T}$ transverse to the $W^s$-foliation. If $\mathcal{T}$ returns to a fixed Pesin block at infinitely many times $n_j$, then the expansion of the inverse maps $\tilde{f}^j:=\tilde f_{x,n_j}$ (as defined in  Lemma \ref{prop:f-under-local-chart}) along $ f^{n_j}(\mathcal{T})$ is bounded by $ce^{j\vep}$ for any  small $\vep$, where $c$ is a uniform constant.  
Projecting onto $\mathcal{T}$ within the Lyapunov neighborhoods, together with the Lipschitz continuity of stable foliation, implies that the transversal direction does not contribute to the entropy.  
\end{itemize}
We remark that the Lipschitz continuity of stable foliation, which holds under $C^2$ smoothness in \cite{LY1, Liu08}, is now known to hold under $C^{1+\alpha}$ smoothness as well \cite{Brown}.  

\medskip
With the establishment of \eqref{conditional_entropy}, the proof is reduced to showing that the equation 
\begin{align*}
H_\mu(\zeta|f^{-1}\zeta)=F_\mu(f)-\sum\nolimits_{\lambda_i<0}\lambda_i 
\end{align*}
implies  $\mu^\zeta_x\ll\lambda^s_x$ for $\mu$-a.e. $x\in M$, i.e., there exists a  measurable function $\psi$ such that for $\mu$-a.e. $x,$ 
\begin{align*}
d\mu^\zeta_x(y)=\psi (y)d\lambda^s_x(y),\quad\text{$\mu^\zeta_x$-a.e. $y\in\zeta(x).$}
\end{align*} 
Note that for $x\in I,$ $\zeta(x)=\{x\}$ and hence $\psi(x)=1.$
{In the following,  
we consider only $x\in\Delta_2.$}

For each $x\in\Delta_2$, denote  
\begin{align*}
\Delta(x,y)= \Theta(x,y) \prod_{k=0}^{\infty}\dfrac{D^s(f^ky)}{D^s(f^kx)} ,\quad\text{$\lambda^s_x$-a.e.\ $y\in\zeta(x)$},
\end{align*}
and define
\begin{eqnarray}\label{def_h}
\psi(y)=\dfrac{\Delta(x,y)}{\int_{\zeta(x)}\Delta(x,z)d\lambda^s_x(z)},\quad\text{$\lambda^s_x$-a.e.}\ y\in\zeta(x).
\end{eqnarray}

\begin{Lem}\label{lem:Delta_transfer}
For $\mu$-a.e.  $x\in\Delta_2$ and $m\ge1,$ it holds that
\begin{eqnarray*}
\int_{(f^{-m}\zeta)(x)}\Delta(x,y)d\lambda^s_x(y)=\dfrac{J_f(x)\cdots J_f(f^{m-1}x)}{D^s(x)\cdots D^s(f^{m-1}x)}\int_{\zeta(f^mx)}\Delta(f^mx,y)d\lambda^s_{f^mx}(y).
\end{eqnarray*}
\end{Lem}
\begin{proof} We only prove the case $m=1;$  the case  $m>1$  is similar.  

For any $\delta>0$ let $V_\delta$ be an open neighborhood of $\Sigma_f$ in $M$ with 
 diameter less than $\delta.$ 
For any  open set $U\subseteq M,$ denote  $U^{-1,c}$   the set of connected
components of the preimage 
$f^{-1}U$ in $M$. 
Given any $y\in\zeta(fx)$ and open neighborhood 
$U_y$ of $y,$ there exists $L_y\ge1$ such that
\begin{align*}
f^{-1}U_{y}
=\big(\cup_{i=1}^{L_y}V_{y^{-1}_i}\big)\cup \tilde V_{\delta,y},
\end{align*}
where 
$\tilde V_{\delta,y}=\cup_{V\in U_y^{-1,c}:V\cap V_\delta\neq\emptyset}V,$ and for each $1\le i\le L_y,$ 
$y^{-1}_i\in f^{-1}\{y\}\backslash V_\delta$ and $V_{y_i^{-1}}\in U^{-1,c}_y$ is an open neighborhood of $y_i$ 
in $M$ such that  $f_{i}:=f|_{V_{y_i^{-1}}}:V_{y_i^{-1}}\to U_{y}$ is a diffeomorphism. Then for any Borel set $B\subseteq U_y\cap \zeta(fx),$ 
\begin{eqnarray*}
&&\int_{f^{-1}B}\Delta(x,z)d\lambda^s_x(z)\\
&=&
\sum_{i=1}^{L_y}\int_{V_{y_i^{-1}}\cap f^{-1}B}\Delta(x,z)d\lambda^s_x(z)+\int_{\tilde V_{\delta,y}\cap f^{-1}B}\Delta(x,z)d\lambda^s_x(z)\\
&=&\sum_{i=1}^{L_y}
\int_{U_y\cap B}\Delta(x,f_i^{-1}y)\cdot(1/D^s(f_i^{-1}y))d\lambda^s_{fx}(y)\\
&&+\int_{\tilde V_{\delta,y}\cap f^{-1}B}\Delta(x,z)d\lambda^s_x(z)\\
&=&\sum_{i=1}^{L_y}\int_{U_y\cap B}\prod_{n=0}^\infty\dfrac{J_f(f^nx)}{J_f(f^n(f_i^{-1}y))}\cdot\dfrac{D^s(f^n(f_i^{-1}y))}{D^s(f^nx)}\cdot\dfrac{1}{D^s(f_i^{-1}y)}d\lambda^s_{fx}(y)\\
&&+\int_{\tilde V_{\delta,y}\cap f^{-1}B}\Delta(x,z)d\lambda^s_x(z)\\
&=&\dfrac{J_f(x)}{D^s(x)}\int_{U_y\cap B}\big(\sum_{i=1}^{L_y}\dfrac{1}{J_f(f_i^{-1}y)}\big)\Delta(fx,y)d\lambda_{fx}^s(y)+\int_{\tilde V_{\delta,y}\cap f^{-1}B}\Delta(x,z)d\lambda^s_x(z). 
\end{eqnarray*}

Note that for any $\delta>0$, we can take a finite partition  
of $\zeta(fx)$ by Borel sets $B_{y_1},...,B_{y_k},$ where each  
$B_{y_j}=U_{y_j}\cap\zeta(fx)$, such that for each $1\le j\le k$, $\tilde V_{\delta,y_j}$ is contained within a neighborhood of $\Sigma_f$ with diameter no greater than $2\delta.$ Hence,
\begin{eqnarray*}
\int_{(f^{-1}\zeta)(x)}\Delta(x,z)d\lambda^s_{x}(z)&=&\dfrac{J_f(x)}{D^s(x)}\sum_{j=1}^k\int_{B_{y_j}}\big(\sum_{i=1}^{L_{y_j}}\dfrac{1}{J_f(f_i^{-1}y)}\big)\Delta(fx,y)d\lambda_{fx}^s(y)\\
&&+\int_{\cup_{j=1}^k\tilde V_{\delta,y_j}\cap
\zeta(x)}\Delta(x,z)d\lambda^s_x(z).
\end{eqnarray*}

By condition (H)'', for $\lambda^s_x$-a.e. $y\in\zeta(fx),$
$\sum_{i=1}^{L_{y}}\dfrac{1}{J_f(f_i^{-1}y)}\to1$ as $\delta\to0,$ and hence 
\[\sum_{j=1}^k\int_{B_{y_j}}\big(\sum_{i=1}^{L_{y_j}}\dfrac{1}{J_f(f_i^{-1}y)}\big)\Delta(fx,y)d\lambda_{fx}^s(y)\to\int_{\zeta(fx)}\Delta(fx,y)d\lambda^s_{fx}(y),\quad \text{as}\ \delta\to0.\]
Also, since $|\det D_zf|=0$ on $\Sigma_f,$ it follows that for any $z\in V_\delta$, $D^s(z)\to0$ as $\delta\to0.$ Thus, by Lemma \ref{prop:bound}, for any $z\in \tilde V_{\delta,y_j}\cap\zeta(x),$
\[\prod_{n=0}^\infty\dfrac{D^s(f^nz)}{D^s(f^nx)}=\dfrac{D^s(z)}{D^s(x)}\prod_{n=1}^\infty\dfrac{D^s(f^nz)}{D^s(f^nx)}\to0\quad\text{as $\delta\to0,$}\] 
which, together with
the integrability of $    \Theta(x,y)=\lim\limits_{n \rightarrow\infty}   \big(\prod_{k=0}^n\frac{J_f(f^kx)}{J_f(f^ky)}\big)\frac{\psi(f^{n+1}y)}{\psi(f^{n+1}x)}  $ in  
condition (H)'', yields
\begin{eqnarray*}
&&\int_{\cup_{j=1}^k\tilde V_{\delta,y_j}\cap
\zeta(x)}\Delta(x,z)d\lambda^s_x(z)\\
&=&\int_{\cup_{j=1}^k\tilde V_{\delta,y_j}\cap
\zeta(x)}  \Theta(x,y)\cdot\prod_{n=0}^\infty\Big(\dfrac{D^s(f^nz)}{D^s(f^nx)}\Big)d\lambda^s_x(z)\to0
\end{eqnarray*}
as $\delta\to0.$
Therefore, 
\begin{eqnarray*}
\int_{(f^{-1}\zeta)(x)}\Delta(x,z)d\lambda^s_x(z)=\dfrac{J_f(x)}{D^s(x)}\int_{\zeta(fx)}\Delta(fx,y)d\lambda^s_{fx}(y).
\end{eqnarray*}

This completes the proof for the case $m=1.$  
The case   $m>1$ follows similarly by replacing $f^{-1}\zeta$ with $f^{-m}\zeta$ in the above arguments. 
\end{proof}

The following considers the inverse limit system $(\bar M,\tau,\bar\mu)$ of $(M,f,\mu)$, and let $\bar\zeta=\pi^{-1}\zeta.$ 
Since, by Lemma \ref{prop:partition_xi},  $\vee_{n=0}^\infty\tau^n\bar\zeta$ is a partition of $\bar M$ into single points, we can define a Borel probability $\bar\nu$ on $\bar M$ 
via the  refining process 
$\{\tau^n\bar\zeta:n\ge1\}.$ 
To be specific, 
consider each element of $\bar\zeta$, which is of form 
$[A_0]=\{\bar y=(\cdots,y_0,\cdots):y_0\in A_0\}$ with $A_0\in\zeta,$ let $\bar\nu([A_0])=\bar\mu([A_0])$.  
Note that for any $\bar x=(\cdots,x_{-1},x_0,x_1,\cdots)\in\bar M$ and any  $p,q\ge0,$ 
the cylinder sets  
\[[A_{-p,...,q}]=\big\{\bar y=(\cdots,y_{-1},y_0,y_1,\cdots):y_{i}\in A_{i},i=-p,...,q\big\}\]
where $A_i\in\zeta$ for $-p\le i\le q$ and  $A_0=\zeta(x_0),$    
is a refinement (the  refinement actually happens for $-p\le i\le0$) of the element 
$\bar\zeta(\bar x)$ of $\bar\zeta$. 
Define the conditional measure of $\bar\nu$ on  $\bar\zeta(\bar x)$ as 
\begin{eqnarray}\label{nu_conditional}
 \bar\nu^{\bar\zeta}_{\bar x}([A_{-p,...,q}])=\dfrac{\int_{\tilde A}\Delta(x_{-p},y)d\lambda^s_{x_{-p}}(y)}{\int_{(f^{-p}\zeta)(x_{-p})}\Delta(x_{-p},y)d\lambda^s_{x_{-p}}(y)},   
\end{eqnarray}
where 
\[\tilde A=\{y\in(f^{-p}\zeta)(x_{-p}):y\in A_{-p},fy\in A_{-p+1},\cdots,f^{p+q}y\in A_q\}.\] 
From the proof of Lemma \ref{lem:Delta_transfer}, we see  that the
$\nu^{\bar\zeta}_{\bar x}$-measure of each cylinder set is  independent of the pull-back time $p$,  ensuring that $\nu^{\bar\zeta}_{\bar x}$ is  well-defined. Thus, the Borel probability 
 measure  $\bar\nu$ on $\bar M$ is  uniquely  determined and is $\tau$-invariant.

\begin{Lem}\label{lem:nu_entropy_eqaulity}
 For any $n\ge1,$
\begin{eqnarray*}
\dfrac{1}{n}\int_{\bar M}-\log\bar\nu^{\bar\zeta}_{\bar x}\big((\tau^n\bar\zeta)(\bar x)\big)d\bar\mu(\bar x)= F_\mu(f)-\sum\nolimits_{i:\lambda_i<0}\lambda_i.
\end{eqnarray*}
\end{Lem}
\begin{proof}  
We adopt the approach of Lemma 4.1.5 in \cite{Liu08} for the integrable setting. 

For any $\bar x=(\cdots,x_{-1},x_0,x_1,\cdots)\in\bar M,$ denote \[g(\bar x)=
\int_{\zeta(x_0)}\Delta(x_0,y)d\lambda^s_{x_0}(y),\]
and for any $n\ge1,$ let  $q_n(\bar x)=\bar\nu^{\bar\zeta}_{\bar x}\big((\tau^n\bar\zeta)(\bar x)\big).$ 
By definition  of $\bar\nu^{\bar\zeta}_{\bar x}$, 
\begin{eqnarray*}
q_n(\bar x)
=\dfrac{g(\tau^{-n}\bar x)}{\int_{(f^{-n}\zeta)(x_{-n})}\Delta(x_{-n},y)d\lambda^s_{x_{-n}}(y)}.
\end{eqnarray*}
Since by Lemma \ref{lem:Delta_transfer}
\begin{eqnarray*}
\int_{(f^{-n}\zeta)(x_{-n})}\Delta(x_{-n},y)d\lambda^s_{x_{-n}}(y)&=&\dfrac{J_f(x_{-n})\cdots J_f(x_{-1})}{D^s(x_{-n})\cdots D^s(x_{-1})}\int_{\zeta(x_0)}\Delta(x_0,y)d\lambda^s_{x_0}(y)\\
&=&\dfrac{J_f(x_{-n})\cdots J_f(x_{-1})}{D^s(x_{-n})\cdots D^s(x_{-1})}g(\bar x),
\end{eqnarray*}
we have \begin{eqnarray*}\label{conditional_nu}
q_n(\bar x)=\dfrac{g(\tau^{-n}\bar x)}{g(\bar x)}\prod_{i=1}^n\dfrac{D^s(x_{-i})}{J_f(x_{-i})}.
\end{eqnarray*}
Note that $q_n(\bar x)\le1.$ Combined with the $\mu$-integrability  of $\log D^s(x)$ and $\log J_f(x),$ we have $\displaystyle\int_{\bar M}\log\dfrac{g(\tau^{-n}\bar x)}{g(\bar x)}d\bar\mu(\bar x)<\infty,$ which, by the $\tau$-invariant of $\bar\mu,$ yields
\[\displaystyle\int_{\bar M}\log\dfrac{g(\tau^{-n}\bar x)}{g(\bar x)}d\bar\mu(\bar x)=0.\]
Therefore,
\begin{eqnarray*}
\int_{\bar M}-\log q_n(\bar x)d\bar\mu(\bar x)&=&\sum_{i=1}^n\int_{\bar M}\log J_f(x_{-i})d\bar\mu(\bar x)-\sum_{i=1}^n\int_{\bar M}\log D^s(x_{-i})d\bar\mu(\bar x)\\
&=&n\int_{M}\log J_f(x)d\mu( x)-n\int_{M}\log D^s(x)d\mu(x),
\end{eqnarray*}
where the second equality follows from the invariance of $\mu$.
Note that 
\[F_\mu(f)=\int\log\mu_x^{f^{-1}\epsilon}(\epsilon(x))=\int_{M}\log J_f(x)d\mu(x),\]
and we hence have 
\begin{eqnarray*}
\dfrac{1}{n}\int_{\bar M}-\log\bar\nu^{\bar\zeta}_{\bar x}\big((\tau^n\bar\zeta)(\bar x)\big)d\bar\mu(\bar x)&=&\dfrac{1}{n}\int_{\bar M}-\log q_n(\bar x)d\bar\mu(\bar x)\\
&=& F_\mu(f)-\sum\nolimits_{\lambda_i<0}\lambda_i. 
\end{eqnarray*}
\end{proof}

We are now prepared  to complete the proof of 
Theorem \ref{thm1}.
\medskip

\noindent{\it  The sufficiency proof of Theorem \ref{thm1}}: 
For any $n\ge1,$ 
\begin{eqnarray*}
H_{\bar\mu}(\tau^n\bar\zeta|\bar\zeta)=H_\mu(\zeta|f^{-n}\zeta)=nh_\mu(f),
\end{eqnarray*}
where the last equality follows from Proposition \ref{prop:entropy}. Hence, if \eqref{eq non-inv of folding type} holds then
\begin{eqnarray*}
\dfrac{1}{n}H_{\bar\mu}(\tau^n\bar\zeta|\bar\zeta)=F_\mu(f)-\sum\nolimits_{\lambda_i<0}\lambda_i.
\end{eqnarray*}
Combined with Lemma \ref{lem:nu_entropy_eqaulity}, this implies  
\begin{eqnarray*}
\int_{\bar M}-\log\bar\nu^{\bar\zeta}_{\bar x}\big((\tau^n\bar\zeta)(\bar x))d\bar\mu(\bar x\big)=H_{\bar\mu}(\tau^n\bar\zeta|\bar\zeta).
\end{eqnarray*}
Since it also holds that
\begin{eqnarray*}
H_{\bar\mu}(\tau^n\bar\zeta|\bar\zeta)=\int_{\bar M}-\log\bar\mu^{\bar\zeta}_{\bar x}\big((\tau^n\bar\zeta)(\bar x))d\bar\mu(\bar x\big),
\end{eqnarray*}
it follows that  
\begin{eqnarray*}
\int_{\bar M}\log\Big(\dfrac{\bar\nu^{\bar\zeta}_{\bar x}\big((\tau^n\bar\zeta)(\bar x))}{\bar\mu^{\bar\zeta}_{\bar x}\big((\tau^n\bar\zeta)(\bar x))}\Big)d\bar\mu(\bar x\big)=0.
\end{eqnarray*}
Let 
$\varphi=\dfrac{d\bar\nu^{\tau^n\bar\zeta}_{\bar x}}{d\bar\mu^{\tau^n\bar\zeta}_{\bar x}}
\Bigg|_{\tau^n\bar\zeta}.$ By the convexity of $\log$, we have 
\begin{eqnarray*}
\int_{\bar M}\log\Big(\dfrac{\bar\nu^{\bar\zeta}_{\bar x}\big((\tau^n\bar\zeta)(\bar x))}{\bar\mu^{\bar\zeta}_{\bar x}\big((\tau^n\bar\zeta)(\bar x))}\Big)d\bar\mu(\bar x)=\int_{\bar M}\log\varphi d\bar\mu\le\log\Big(\int_{\bar M}\varphi d\bar\mu\Big)=0
\end{eqnarray*}
in which the inequality ``$\le$" becomes  an equality if and only if $\varphi\equiv1$  on $\tau^n\bar\zeta(\bar x)$ for $\bar\mu$-a.e. $\bar x$, that is,
$\bar\nu=\bar\mu$ on $\mathcal B(\tau^n\bar\zeta).$

The above argument applies for all $n\ge1$. Combined with the fact that  $\vee_{n=0}^\infty\tau^n\bar\zeta$ partitions $\bar M$ into single points,
it follows that
$\bar\nu=\bar\mu,$ and hence $\mu=\pi\bar\mu=\pi\bar\nu.$ Therefore, by the construction of $\bar\nu,$
$\mu$ has absolutely continuous conditional measures on the stable manifolds. \qed

\section{Applications}\label{app}

\subsection{Conservative systems using the ``only if " part} Let  $M$ be a Riemannian manifold without boundary and equipped with the volume measure, denoted by $\vol$, induced by its Riemannian structure. Theorem  \ref{thm1} applies to any  volume-preserving map $f$ on $M$, provided that 
\begin{eqnarray}\label{log_det}
\int \log |\det(D_xf)| d\vol<\infty.
\end{eqnarray}
The  non-uniformly hyperbolic  volume-preserving endomorphism  constructed by Andersson, Carrasco, and Saghin in \cite{ACS} is such an example.  Let $f$ be a general volume-preserving map $f$ on $M$ satisfying (\ref{log_det}). 
Since $\vol$ admits absolutely continuous conditional measures on the stable manifolds, the  ``only if'' part  of Theorem \ref{thm1} implies that the entropy formula of folding type  
\[h_{\vol}(f)= F_{\vol}(f)-\displaystyle\int\sum\nolimits_{\lambda_i(x)<0} \lambda_i(x)d\vol(x).
\]
holds, and that the condition (H)'' is satisfied.

As a direct corollary of Remark \ref{Rem1}, the following holds. 
\begin{Prop}
Let $f:M\to M$ be a  volume-preserving map  such that \eqref{log_det} holds. Then the entropy production $e_{\vol}(f)=0.$
\end{Prop}

\subsection{Dissipative systems  using  the ``if" part}
Define  the skew-product
\[f(x,y)=(1-ax^2,  \, g(x,y)): I\times\mathbb T^2\to I\times\mathbb T^2,\]
where $I=[-1,1]$, $1<a\le 2$, and $g: I\times \mathbb{T}^2\to  \mathbb{ T}^2$ satisfies
\begin{eqnarray*}
 g(x,y)= g_1(y)+g_2(x,y),  
\end{eqnarray*}
with $g_1:\mathbb{T}^2\to\mathbb T^2$ an Anosov linear endomorphism and $g_2 : I \times   \mathbb{T}^2 \to  \mathbb{T}^2$ being $C^{1+\alpha}$-close to zero. The SRB property of the first component of $f,$ denoted by
\begin{eqnarray*}
  f_1(x):=1-ax^2, 
\end{eqnarray*}
exhibits degeneracy and has been studied by
Misiurewicz  \cite{Mis81}  and  Jakobson  \cite{Jak}. Note that the map $f$  is a skew-product  over $f_1$. 
In the following,  the inverse SRB property of $f$ is investigated. 

\begin{Prop}\label{project}There exists an  inverse SRB measure    $\mu$   of $f$ such that the  entropy formula  of folding type holds:
\[h_{\mu}(f)= F_{\mu}(f)-\displaystyle\int\lambda^-(x)d\mu(x),
\] 
where $\lambda^-(x)$ denotes the unique negative Lyapunov exponent at  $x$ for $\mu$-a.e. $x$. 
\end{Prop}

\begin{proof}  Note that the map $f_1(x)=1-ax^2$ is piecewise monotone on $I$. Let $c(x)=\sharp f_1^{-1}(x)$,     where $\sharp A$ denotes the cardinality of the set $A$.       It is known that $f_1$ admits  an ergodic fair measure $\mu_{1}$ with  positive   metric entropy (\cite{RZ, Mis})  \[h_{\mu_{1}}(f_1)=\int \log c(x)d\mu_{1}(x).\]     
	
Let $r$ denote the degree of $g_1$. Since $g_2$ is close to zero, for each fixed $x\in I$,  every $y\in\mathbb{T}^2$ has $r$ preimages under the map $g_x:=g(x, \cdot):\, \mathbb{T}^2 \to  \mathbb{T}^2$. 
Let $\Phi^s= \log |\det(Df\mid_{E^s})|$, where $E^s$ denotes the Oseledets splitting corresponding to the negative Lyapunov exponents. Note that $\Phi^s$ is a  continuous function on $I\times \mathbb{T}^2$. 

Let $D\subset I$ be a subset of $I$. For any $\vep>0$ and $n\in \mathbb{N}$, a subset   $E\subset D$ is called an $(n,\vep)$-separated set  of $f_1$ on $D$ if  for any $x\neq y\in E$, there exists $0\le i\le n-1$ such that $d(f^i_1(x), f^i_1(y))>\vep$.  Let $S_n(f_1,D,\vep)$ denote an $(n,\vep)$-separated set  of $f_1$ with respect to  $D$ that has  maximal cardinality.    
In the same manner, a subset $F\subset \{x\}\times \mathbb{T}^2$ is called an $(f,x,n,\vep)$-separated set at $x$  along the fiber $\mathbb{T}^2$, if for any $y\neq z\in F$, there exists  $0\le i\le n-1$ such that $$d\big(g_{f_1^{i}(x)}\circ \cdots \circ g_{f_1(x)}\circ g_{x}(y), \,\, g_{f_1^{i}(x)}\circ \cdots \circ g_{f_1(x)}\circ g_{x}(z)\big)>\vep.$$  
Given any continuous function $\phi:I\times\mathbb T^2\to\mathbb R$ and $n\ge1,$ define 
\begin{eqnarray}\label{P}
P_n(f, x,\vep, \phi)=\max_{\substack{F\subset \mathbb{T}^2:(f,x,n,\vep)\\\text{-separated set}}} \,\,\, \sum\nolimits_{y\in F} e^{\phi_n(x,y)},  
\end{eqnarray}
where $\phi_n(x,y)=\sum_{j=0}^{n-1}\phi(f^i(x,y))$. Define
$$P_n(f, D, \vep, \phi)=\sum\nolimits_{x\in S_n(f_1,D, \vep)} P_n(f, x,\vep, \phi),$$ and 
$$P(f, D, \vep, \phi)=\limsup_{n\to \infty} \frac{1}{n} \log P_n(f, D, \vep, \phi). 
$$
	
	Let $\{\psi_1,\psi_2,\cdots\}$ be a dense subset of the space of continuous functions on $I$, and define $A_0=I, N_0=0$. For each $k\ge 1$, there exists a  subset $A_k\subset A_{k-1}$  with $\mu_1(A_k)>\mu_1(A_{k-1})-\frac{1}{2^{k+1}}$ and integers $N_k>N_{k-1}$ such that for any $x\in A_k,$
	$$\Big|\frac{1}{n}\sum_{ i=0}^{n-1}\psi_j(f_{1}^{i}(x))-\int \psi_j d\mu_1\Big|<\frac{1}{2^k},\quad \forall\,n\ge N_k,\,\, 1\le j\le k.$$
  	Let $A=\cap_{k\ge 1} A_k$. Then one has  $$\mu_1(A)\ge 1-\sum\nolimits_{k\ge 1}\frac{1}{2^{k+1}}=\frac{1}{2}.$$
	    
Henceforth, take $\phi=\Phi^s-\log r$ and $D=A$. For any $n\ge 1$ and   $\vep>0$, 
let $S_n(f, x, \vep,\phi)$ denote an $(f,x,n,\vep)$-separated set that attains the maximal value in \eqref{P}. 
Define  $$\nu_{n,\vep}=\frac{1}{P_n(f, A, \vep, \phi)} \sum\nolimits_{x\in S_n(f_1,A,\vep)}\sum\nolimits_{y\in S_n(f, x, \vep,\phi)} e^{\phi_n(x,y)}\delta_{(x,y)},$$
and $$\mu_{n,\vep}=\frac{1}{n}\sum\nolimits_{i=0}^{n-1}(f^i)_* \nu_{n,\vep},$$
where $(f^i)_* \nu_{n,\vep}=\nu_{n,\vep}\circ f^{-i}$. Note that both $\nu_{n,\varepsilon}$ and $\mu_{n,\varepsilon}$ are probability measures on $I\times\mathbb T^2$. 

Let $\pi: I\times \mathbb{T}^2 \to I$ be the projection $\pi(x,y)=x$. Note  that 
$$\pi_*(\nu_{n,\vep})=\frac{1}{P_n(f, A, \vep, \phi)} \sum\nolimits_{x\in S_n(f_1, A,\vep) } \sum\nolimits_{y\in S_n(f, x, \vep,\phi)}e^{\phi_n(x,y)}\delta_x.$$
Since $\pi\circ f=f_1\circ \pi$, it follows that 
$$\pi_*(\mu_{n,\vep})=\frac{1}{n} \sum_{i=0}^{n-1}\pi_* \circ f_*(\nu_{n,\vep})=\frac{1}{n} \sum_{i=0}^{n-1} (f_1)_*\circ\pi_*(\nu_{n,\vep}).$$ 
	By the choice of $A$,  for any $\psi_j$ and $k\ge j$,  there exists $N_k\in\N$ such that for any $n\ge N_k$,  
	$$\big|\int \psi_j d\pi_*(\mu_{n,\vep})-\int \psi_jd\mu_1\big|<\frac{1}{2^k},$$
	which implies that 
    $$\lim_{n\to \infty}\int \psi_j d\pi_*(\mu_{n,\vep})=\int \psi_jd\mu_1.$$
	By the arbitrariness of $\psi_j$, it follows that for any continuous function $\psi: I\to \mathbb{R}$, 
	$$\lim_{n\to \infty}\int \psi d\pi_*(\mu_{n,\vep})=\int \psi d\mu,$$
that is, $\lim_{n\to \infty}\pi_*(\mu_{n,\vep})=\mu_1.$

Given any $n\ge1$ and 	$\varepsilon>0$, let $L_n(\varepsilon)$ denote the
minimal number of $\varepsilon$-balls  in the
$d_n$-metric required to cover a set of $\mu_1$-measure at least $1/2$, where 
$$d_n(x,z):=\max\nolimits_{0\leq i\leq n-1}d(f^i(x),\,f^i(z)),\quad\forall x,z\in I.$$ 
Denote 
$$h_{\mu_1}(f_1,\vep)=\liminf_{n\rightarrow\infty}\dfrac{1}{n}\log L_n(\varepsilon).$$ 
By Katok's equivalent definition of metric entropy \cite{Katok2}, 
it satisfies that 
$$h_{\mu_1}(f_1)=\lim_{\varepsilon\rightarrow0}h_{\mu_1}(f_1,\vep).$$ 
For any small $\delta>0$, one can take $0<\vep<\delta$ such that  $$h_{\mu_1}(f_1,\vep)>h_{\mu_1}(f_1)-\delta.$$	

Note that $\{(x,y)\in I\times\mathbb T:x\in S_n(f_1,A,\varepsilon),y\in S_n(f,x,\varepsilon,\phi)\}$
is an $(n,\varepsilon)$-separated set of $f$. By the arguments of variational principle (see Page 219-221 of \cite{Walters}),  for any $0<\vep<\delta$, there exists an $f$-invariant measure $\mu_{\vep}$  such that  
	\begin{eqnarray}\label{pres} P_{\mu_{\vep}}(f, \phi):=h_{\mu_{\vep}}(f)+\int \phi d\mu_{\vep}\ge P(f,A, \vep, \phi). \end{eqnarray}
Moreover, according to the estimates  in Theorem 1 of \cite{Mil} (the end of Page 806 and the second to last paragraph  of the proof of  Theorem 1),  there exist $A_{\vep}, B_{\vep}>0$ such that
	$$ A_{\vep} r^n \Leb(\mathbb{T}^2)\le P_n(f, x,\vep,\Phi^s) \le B_{\vep} r^n \Leb(\mathbb{T}^2),$$
	and since $\phi=\Phi^s-\log r,$ it follows that
    $$ A_{\vep}  \Leb(\mathbb{T}^2)\le P_n(f, x,\vep,\phi) \le B_{\vep} \Leb(\mathbb{T}^2).$$
	Since $\bigcup_{x\in S_n(f_1,A, \vep)} B_n(x,\vep) \supseteq A$ and $\mu_1(A)\ge \frac12$, we have
	\[L_n(\vep)\le  \sharp S_n(f_1,A,\vep).\] 
    Therefore,
    \begin{eqnarray*}P(f,A, \vep,\phi)&=&\limsup_{n\to \infty} \frac{1}{n} \log \sum\nolimits_{x\in S_n(f_1,A, \vep)} P_n(f,x,\vep,\phi)\\[2mm] &=& \limsup_{n\to \infty} \frac{1}{n} \log \sharp S_n(f_1,A,\vep)\ge h_{\mu_1}(f_1,\vep) >h_{\mu_1}(f_1)-\delta.\end{eqnarray*}
	Therefore, combined with (\ref{pres}), 	
	\begin{eqnarray}\label{entropy}
		h_{\mu_{\vep}}(f)+\int \phi d\mu_{\vep}\ge  h_{\mu_1}(f_1)-\delta.
	\end{eqnarray}
	Letting $\delta\to 0$ and passing to a subsequence $\varepsilon_n$, one obtains an $f$-invariant   $\mu$ such that   
    $\mu_{\vep_n}\to\mu$ as $n\to\infty.$ Since $\pi_*(\mu_{\vep_n})=\mu_1$, it follows that $\pi_*(\mu)=\mu_1$.   

In the following, it is shown that 
$\mu$ satisfies the entropy formula of folding type \eqref{eq non-inv of folding type} and the condition  (H)''. 
It then follows from Theorem \ref{thm1}  that $\mu$ is an inverse SRB measure.

    Since $f_1$ is  piecewise monotone, it is asymptotically entropy expansive \cite{MS}. Moreover, as $f$ is uniformly hyperbolic along the fiber $\mathbb{T}^2$, it follows that $f$ is also asymptotically entropy expansive. Thus, the metric entropy of $f$ is upper semi-continuous \cite{Mis1976}. Therefore, by \eqref{entropy}, 
	\begin{eqnarray*}
		h_{\mu}(f)+\int \phi d\mu\ge h_{\mu_1}(f_1).
	\end{eqnarray*}
	
We observe that for the ergodic measure $\mu_1$, it holds that $h_{\mu_1}(f_1)=F_{\mu_1}(f_1)$. In fact, by Lemma \ref{lem:conditional-entropy-upper-bound},    $h_{\mu_1}(f_1)\ge F_{\mu_1}(f_1)$. Since $\mu_1$  has only non-negative Lyapunov exponent, the folding-type Ruelle inequality implies $h_{\mu_1}(f)\le F_{\mu_1}(f_1).$ Thus, $h_{\mu_1}(f_1)= F_{\mu_1}(f_1)$ follows. Therefore,
\begin{eqnarray}\label{large}
h_{\mu}(f)\ge  h_{\mu_1}(f_1)-\int \phi d\mu=  F_{\mu_1}(f_1)+\log r - \int \Phi ^sd\mu.
\end{eqnarray}
Since $\pi_*(\mu)=\mu_1$,   it follows that  
	\begin{eqnarray}\label{one side}
	\label{small 1}H_{\mu}(\epsilon \mid f^{-1}\epsilon) &\le &  \int 	\log (r c(f_1(x))) d\mu(x,y)\\
	&=&\log r +  \int 	\log (c(f_1(x))) d\mu(x,y)\nonumber  \\&=&   \log r + \int 	\log (c(f_1(x))) d\mu_1(x)\nonumber\\
	&=&  \log r + F_{\mu_1}(f_1). \nonumber 
	\end{eqnarray}
That is, 
	\begin{eqnarray}\label{less than}
	F_{\mu}(f)\le 	 \log r +F_{\mu_1}(f_1).
	\end{eqnarray}
Combining \eqref{less than} and   (\ref{large}), one obtains \begin{eqnarray*}
		h_{\mu}(f)\ge  F_{\mu}(f)- \int \Phi ^sd\mu=F_{\mu}(f)- \int \sum\nolimits_{\lambda_i(x)<0}\lambda_i(x)d\mu .
	\end{eqnarray*}
	On the other hand, by the folding-type Ruelle inequality, 
	\begin{eqnarray}\label{Ruelle}
		h_{\mu}(f)\le  F_{\mu}(f)-  \int \sum\nolimits_{\lambda_i(x)<0}\lambda_i(x)d\mu .
	\end{eqnarray}
	Therefore, the equality  
    \begin{eqnarray*}
		h_{\mu}(f)= F_{\mu}(f) - \int \sum\nolimits_{\lambda_i(x)<0}\lambda_i(x)d\mu
	\end{eqnarray*}
    holds.     
Note that for $\mu$-a.e. $x,$ the three Lyapunov exponents of $f$ include exactly one negative exponent, which we denote by $\lambda^-(x).$ Consequently, 
\begin{eqnarray*}
 h_{\mu}(f)= F_{\mu}(f)-\displaystyle\int\lambda^-(x)d\mu(x).  
\end{eqnarray*}

It remains to verify that condition (H)'' is satisfied. Combining (\label (\ref{large}) and (\ref{Ruelle}), we deduce 
		\begin{eqnarray*} F_{\mu}(f)\ge 	F_{\mu_1}(f_1)+ \log r,\end{eqnarray*}
		which, together with (\ref{less than}), yields
		\begin{eqnarray*} F_{\mu}(f)=	F_{\mu_1}(f_1)+ \log r.\end{eqnarray*}
	This implies that the inequality in (\ref{small 1}) holds as an equality. 
    Therefore,  for $\mu$-a.e. $(x,y)\in I\times  \mathbb{T}^2$,
    \begin{eqnarray}\label{J}
     \frac{1}{J_{f,\mu}(x,y)}=\mu^{f^{-1}\epsilon}(\epsilon(x,y))=\frac{1}{r c(f_1(x))}=\frac{1}{r J_{f_1, \mu_1}(x)}.   
    \end{eqnarray}
	Note that \eqref{J} is independent of $y$. 
	Since the system has a skew product structure, the stable manifold $W^s(x,y)\subset \{x\}\times \mathbb{T}^2$. Thus, for $\mu_1$-a.e. $x\in I,$    
    one can define  
    \begin{eqnarray*}
     J_{f}(x,y)=r J_{f_1, \mu_1}(x)\quad \text{on $W^s(x,y)$}   
    \end{eqnarray*}
     Therefore, for any $(\tilde{x}, \tilde{y})\in  W^s(x,y)$, it holds that    $$\prod_{n=0}^\infty\dfrac{J_{f}(f^n(x,y))}{J_{f}(f^n(\tilde{x}, \tilde{y}))}=1$$
	Thus, condition (H)"  is satisfied by Remark \ref{rem:generalization}.    
	
	Now,  the $f$-invariant measure $\mu$ satisfies both  the entropy formula of folding type and the condition (H)". By the 
    ``if"  part  of Theorem  \ref{thm1},     $\mu$ is  an inverse SRB measure, i.e.,  it  has absolutely continuous conditional measures on the stable manifolds. 	
\end{proof}

\end{document}